\documentclass{article}[11pt,a4paper]

\usepackage[frenchb, english]{babel}
\usepackage{amsmath}
\usepackage{amsfonts}
\usepackage{amsthm}
\usepackage{amssymb}
\usepackage{bm}
\usepackage[active]{srcltx}

\usepackage[latin1]{inputenc}

\usepackage{enumerate}
\usepackage{multicol}

\usepackage[all,cmtip]{xy}

\newtheoremstyle{mystyle2}{}{}{}{2pt}{\scshape}{.}{ }{}
\newtheoremstyle{mystyle}{}{}{\slshape}{2pt}{\scshape}{.}{ }{}
\newtheoremstyle{etapestyle}{}{}{\itshape}{2em}{\sffamily}{:}{ }{\thmname{#1}}
\newtheoremstyle{definitionstyle}{}{}{}{2pt}{\bfseries}{.}{ }{}

\newtheorem{thm}{Th\'{e}or\`{e}me}[section]
\newtheorem{cor}[thm]{Corollaire}

\newtheorem{prop}[thm]{Proposition}

\newtheorem{lemme}[thm]{Lemme}

\newtheorem*{prop*}{Proposition}

\theoremstyle{mystyle2}

\theoremstyle{mystyle}

\theoremstyle{remark}
\newtheorem{rem}[thm]{Remarque}
\newtheorem{question}[thm]{Question}

\theoremstyle{etapestyle}

\theoremstyle{definitionstyle}

 \DeclareMathOperator{\car}{car}
 \DeclareMathOperator{\Spec}{Spec}
 
 \DeclareMathOperator{\Coker}{Coker}

\DeclareMathOperator{\Ker}{Ker} 
\DeclareMathOperator{\Pic}{Pic}

\DeclareMathOperator{\Hilb}{Hilb}
\DeclareMathOperator{\Ima}{Im}
\DeclareMathOperator{\Id}{Id}

\DeclareMathOperator{\pfaff}{pfaff}
\DeclareMathOperator{\Aut}{Aut}

\newenvironment{msc}{\begin{center}
\begin{minipage}[c]{11cm} {\bf MSC classification:}} {\end{minipage}
\end{center} 
}

\begin{document}

\title{S\'eparation et propri\'et\'e de Deligne-Mumford des champs de modules d'intersections compl\`etes lisses}
\author{Olivier BENOIST
\footnote{DMA, ENS, 45 rue d'Ulm, 75005 Paris, FRANCE, {\ttfamily olivier.benoist@ens.fr}}
}
\date{}
\maketitle
\renewcommand{\abstractname}{R\'esum\'e}
\begin{abstract} 
On montre que les champs de modules d'intersections compl\`etes lisses dans $\mathbb{P}^N$ polaris\'ees par $\mathcal{O}(1)$ sont s\'epar\'es
(sauf dans le cas des quadriques) et de Deligne-Mumford (sauf pour quelques exceptions). 
 \end{abstract}
   \selectlanguage{english}
\renewcommand{\abstractname}{Abstract}
\begin{abstract} 
We show that the moduli stacks of smooth complete
intersections in $\mathbb{P}^N$ polarized by $\mathcal{O}(1)$ are separated (except for quadrics) and Deligne-Mumford (apart from a few exceptions).\end{abstract}
\selectlanguage{frenchb}
\begin{msc}
14D23, 14M10, 14J50
\end{msc}




\section{Introduction}

On commence par discuter la question g\'en\'erale
de la s\'eparation d'un champ de modules de vari\'et\'es polaris\'ees, avant de se restreindre au cas
particulier abord\'e dans cet article : le champ de modules des intersections compl\`etes lisses polaris\'ees par $\mathcal{O}(1)$.

\subsection{S\'eparation d'espaces de modules}

  Soit $\mathcal{M}$ un champ de modules de vari\'et\'es projectives lisses polaris\'ees. 
On s'int\'e\-resse \`a la s\'eparation du champ $\mathcal{M}$. Un r\'esultat classique est le th\'eor\`eme de Matsusaka et Mumford (\cite{MM} Theorem 2):
\begin{thm}\label{MM}
 Si les vari\'et\'es que $\mathcal{M}$ param\`etre ne sont pas birationnellement r\'egl\'ees (i.e. ne sont pas birationnelles \`a une vari\'et\'e
de la forme $\mathbb{P}^1\times Y$), le champ $\mathcal{M}$ est s\'epar\'e. 
\end{thm}

  On peut chercher \`a \'etendre ce r\'esultat \`a d'autres classes de vari\'et\'es. 
Commen\c{c}ons par une remarque facile et utile. Les stabilisateurs des points g\'eom\'e\-triques de $\mathcal{M}$ s'identifient aux groupes d'automorphismes
des vari\'et\'es polaris\'ees param\'etr\'ees par $\mathcal{M}$, qui sont des sch\'emas en groupes affines. Si $\mathcal{M}$ est s\'epar\'e, ces 
stabilisateurs doivent \^etre propres ; comme ils sont donc propres et affines, ils sont m\^eme finis.
On a montr\'e :
\begin{rem}\label{remautofini}
 Une condition n\'ecessaire pour que $\mathcal{M}$ soit s\'epar\'e est que les vari\'et\'es polaris\'ees que $\mathcal{M}$ param\`etre aient des groupes
d'automorphismes finis.
\end{rem}


  Un cas particulier int\'eressant est celui o\`u $\mathcal{M}$ param\`etre des vari\'et\'es de Fano.
Ce cas particulier est motiv\'e par l'intervention en th\'eorie de Mori
de fibrations \`a fibres vari\'et\'es de Fano, donc de familles de vari\'et\'es de Fano.

\vspace{1em}

Regardons des exemples de petite dimension.
En dimension $1$, la seule vari\'et\'e de Fano est $\mathbb{P}^1$, dont le groupe d'automorphismes est de dimension $>0$ ;
par la remarque \ref{remautofini}, la situation n'est pas int\'eressante. 

En dimension $2$, les vari\'et\'es de Fano sont les surfaces de del Pezzo. Notons $\mathcal{M}$ le champ de modules des surfaces de del Pezzo de degr\'e $d$.
Au vu de la remarque \ref{remautofini}, la question de la s\'eparation de $\mathcal{M}$ est int\'eressante si $1\leq d\leq 5$.
Dans ce cas, l'article \cite{Ishii} propose une construction d'un espace de modules grossier pour $\mathcal{M}$
\`a l'aide de th\'eorie g\'eom\'etrique des invariants, qui montre que le champ
$\mathcal{M}$ est s\'epar\'e (appliquer \cite{GIT} Corollary 2.5).

En dimension $3$, on peut parfois appliquer le th\'eor\`eme \ref{MM} de Matsusaka et Mumford via le corollaire suivant :
\begin{cor}\label{sepFano3}
En caract\'eristique nulle, si $\mathcal{M}$ param\`etre des
solides de Fano qui ne sont pas rationnels, $\mathcal{M}$ est n\'ecessairement s\'epar\'e.
\end{cor}
\begin{proof}[$\mathbf{Preuve}$]
Soit $X$ un solide de Fano birationnel \`a $\mathbb{P}^1\times Y$. Comme $X$ est de Fano et que nous sommes en caract\'eristique nulle,
$X$ est rationnellement connexe,
de sorte que $Y$ est une surface rationnellement connexe.
 Comme $Y$ est recouverte par des courbes rationnelles tr\`es libres, toute forme pluricanonique sur $Y$
est nulle. La classification des surfaces montre alors que
$Y$ est birationnelle \`a $\mathbb{P}^1\times C$ pour $C$ une courbe. Comme $Y$ est rationnellement connexe, $C$ est rationnelle, de sorte que $Y$,
donc $X$ sont rationnels. On a montr\'e, comme voulu, qu'un solide de Fano non rationnel n'est pas birationnellement r\'egl\'e.
\end{proof}

Pour certaines collections de solides de Fano \`a groupes d'automorphismes finis,
par exemple celle \'etudi\'ee dans \cite{DIM},
la question de la s\'eparation de $\mathcal{M}$ est ouverte (on sait seulement, par \cite{Beauprym} th\'eor\`eme 5.6,
qu'une vari\'et\'e g\'en\'erale dans cette famille n'est pas rationnelle, de sorte que le corollaire \ref{sepFano3} ne s'applique pas) :
\begin{question}\label{qsepDIM}
 En caract\'eristique nulle, le champ de modules $\mathcal{M}$ des solides de Fano
de nombre de Picard $1$, d'indice $1$ et de degr\'e $10$ est-il s\'epar\'e ?
\end{question}

\'Enon\c{c}ons la question g\'en\'erale que ces exemples illustrent :
\begin{question}\label{qsepFano}
 Quand un champ de modules de vari\'et\'es de Fano \`a groupes d'automorphismes finis est-il s\'epar\'e ?
\end{question}

L'objectif de ce texte est de r\'epondre \`a
la question \ref{qsepFano} dans le cas particulier o\`u $\mathcal{M}$ param\`etre des intersections compl\`etes lisses.

\subsection{\'Enonc\'es des th\'eor\`emes}

Dans tout ce texte, on travaille sur $\Spec(\mathbb{Z})$, et on fixe $N\geq 2$, $1\leq c\leq N-1$ et $2\leq d_1\leq\dots\leq d_c$ des entiers.
Par intersection compl\`ete (sur un corps $k$), on voudra dire sous-sch\'ema ferm\'e de codimension $c$ dans $\mathbb{P}^N_k$
d\'efini par $c$ \'equations homog\`enes de degr\'es $d_1,\dots,d_c$. On notera $n=N-c\geq 1$ la dimension de ces intersections compl\`etes.

Soit $H$ l'ouvert du sch\'ema de Hilbert de $\mathbb{P}^N_{\mathbb{Z}}$ param\'etrant les intersections compl\`etes lisses
(voir par exemple \cite{Sernesi} 4.6.1).
On note $\mathcal{M}$ le champ de modules des intersections
compl\`etes lisses polaris\'ees par $\mathcal{O}(1)$. Par d\'efinition, c'est le champ quotient $[PGL_{N+1}\backslash H]$.

Remarquons que, si $d_1+\ldots+d_c<N+1$, les vari\'et\'es consid\'er\'ees sont de Fano. Pour ces valeurs des param\`etres, on \'etudie donc un
cas particulier de la question \ref{qsepFano}.

\vspace{1em}

  Le premier th\'eor\`eme principal est :
               
\begin{thm}\label{DMintro}
Le champ $\mathcal{M}$ est de Deligne-Mumford sauf dans les cas suivants :

\begin{enumerate}[(i)]
 \item Si $c=1$ et $d_1=2$.

\item Si $N=2$, $c=1$, $d_1=3$, auquel cas il est de Deligne-Mumford au-dessus de $\Spec(\mathbb{Z}[\frac{1}{3}])$.

\item Si $N\geq 3$ est impair, $c=2$, $d_1=d_2=2$, auquel cas il est de Deligne-Mumford au-dessus de $\Spec(\mathbb{Z}[\frac{1}{2}])$.
\end{enumerate}
\end{thm}

  Ce th\'eor\`eme signifie que, sauf pour quelques exceptions, les automorphismes
projectifs d'une intersection compl\`ete lisse forment un sch\'ema en groupes fini r\'eduit. Il y a deux types d'exceptions diff\'erents : dans le cas
(i), ces groupes sont de dimension $>0$ ; dans les cas (ii) et (iii), ils sont finis non r\'eduits.

  Au vu de la remarque \ref{remautofini}, cet \'enonc\'e peut
\^etre vu comme un r\'esultat facile \`a tester indiquant que l'\'etude de la s\'eparation de $\mathcal{M}$
est int\'eressante --- sauf si $c=1$ et $d_1=2$.
C'est ce qui a motiv\'e pour nous son \'etude.
Ce n'est cependant pas un \'enonc\'e plus faible que la s\'eparation du champ
$\mathcal{M}$ car celle-ci ne dit rien sur le caract\`ere r\'eduit de ces groupes d'automorphismes.

\vspace{1em}
 
Le champ de modules des quadriques ne peut \^etre s\'epar\'e par la remarque \ref{remautofini}
car le groupe de leurs automorphismes projectifs est de dimension $>0$.
On montre que ce contre-exemple trivial est le seul :

\begin{thm}\label{sepintro}
  Le champ $\mathcal{M}$ est s\'epar\'e, sauf si $c=1$ et $d_1=2$.
\end{thm}

Ce th\'eor\`eme permet d'appliquer le th\'eor\`eme de Keel et Mori \cite{KM}, pour obtenir :

\begin{cor}\label{edmgrossierintro}
  Si l'on n'a pas $c=1$ et $d_1=2$, le champ $\mathcal{M}$ admet un espace de modules grossier $M$ qui est un espace alg\'ebrique s\'epar\'e.
\end{cor}

Il est alors naturel d'\'etudier l'espace alg\'ebrique $M$ : est-ce un sch\'ema, un sch\'ema quasi-projectif ? Cette question est discut\'ee, et r\'esolue
dans des cas particuliers dans \cite{Oolqp}.

\vspace{1em}

Les deux parties de ce texte sont consacr\'ees aux preuves respectives des th\'eor\`emes \ref{sepintro} (th\'eor\`eme \ref{principal}) et
\ref{DMintro} (th\'eor\`eme \ref{icDM}). On renvoie \`a ces parties pour une discussion
plus pr\'ecise de ces \'enonc\'es, des cas particuliers d\'ej\`a connus...



Remarquons que, contrairement \`a ce qu'on a laiss\'e entendre plus haut, et qui aurait \'et\'e plus naturel,
l'\'etude de la s\'eparation de $\mathcal{M}$ pr\'ec\`ede ici celle des groupes d'automorphismes. La raison pour cela est que,
faute d'argument plus simple, on d\'eduira le th\'eor\`eme \ref{autoic}
quand $N\geq 5$ est impair, $c=2$, $d_1=d_2=2$ et $\car(k)=2$ de la s\'eparation de $\mathcal{M}$.

\section{S\'eparation}

Introduisons quelques notations. Un trait $T$ est le spectre d'un anneau de valuation discr\`ete.
Si $T$ est un trait, on notera toujours $\eta$ son point g\'en\'erique et $s$ son point sp\'ecial,
$R$ l'anneau de valuation discr\`ete dont il est le spectre, $K$ le corps de fractions de $R$, $k$ son corps r\'esiduel, $t$ une uniformisante,
 $v$ la valuation et $\pi:R\to k$ la sp\'ecialisation.

\subsection{Introduction}

\subsubsection{\'Enonc\'e du th\'eor\`eme}

On va montrer dans cette partie le th\'eor\`eme suivant :

\begin{thm}\label{principal}
  Le champ $\mathcal{M}$ est s\'epar\'e, sauf si $c=1$ et $d_1=2$.
\end{thm}
  On peut reformuler plus concr\`etement le th\'eor\`eme \ref{principal}. C'est sous la derni\`ere forme (iii) que nous le d\'emontrerons.

\begin{lemme}
 Les assertions suivantes sont \'equivalentes :
\begin{enumerate}[(i)]
 \item Le champ $\mathcal{M}$ est s\'epar\'e.
 \item Le groupe $PGL_{N+1}$ agit proprement sur $H$.
 \item Pour tout trait $T$ \`a corps r\'esiduel alg\'ebriquement clos, si $Z,Z'\subset \mathbb{P}^N_T$ sont des sous-$T$-sch\'emas ferm\'es 
plats sur $T$
dont les fibres g\'eom\'etriques sont des intersections compl\`etes lisses, tout automorphisme
$f_\eta : \mathbb{P}^N_{\eta} \to \mathbb{P}^N_{\eta}$ tel que $f_\eta(Z_{\eta})=Z'_{\eta}$ se prolonge en un 
automorphisme $f:\mathbb{P}^N_T\to\mathbb{P}^N_T$ tel que $f(Z)=Z'$.
\end{enumerate}
\end{lemme}

\begin{proof}[$\mathbf{Preuve}$]
Consid\'erons le diagramme $2$-cart\'esien
ci-dessous, o\`u l'on a not\'e $\Delta$ le morphisme diagonal de $\mathcal{M}$ et $P:H\to\mathcal{M}$ la pr\'esentation canonique de $\mathcal{M}$ :
\begin{equation}\label{diagonalecartesien}
\xymatrix  {
PGL_{N+1}\times H\ar[d]^{p_2\circ P}\ar[r]^{\hspace{1em}(\sigma,p_2)}  & H\times H\ar[d]^{(P,P)}
\\
\mathcal{M}\ar[r]^{\Delta}&\mathcal{M}\times \mathcal{M}}
\end{equation}
Comme $P$ est lisse et surjectif, $(P,P)$ est \'egalement lisse et surjectif, donc un recouvrement fppf.
Par cons\'equent, par \cite{LMB} 7.11.1, $(\sigma,p_2)$ est propre 
si et seulement si $\Delta$ l'est. Enfin, par \cite{LMB} 7.7, $\Delta$ est propre
si et seulement si $\mathcal{M}$ est s\'epar\'e.
Cela montre l'\'equivalence entre les deux premi\`eres assertions. 

L'\'equivalence entre les deux derni\`eres
est exactement le crit\`ere valuatif de propret\'e appliqu\'e au morphisme $PGL_{N+1}\times H\to H\times H$ 
(on peut se restreindre aux corps r\'esiduels alg\'ebriquement clos par la remarque 7.3.9 (i) de \cite{EGA2}). 
\end{proof}

On a d\'ej\`a vu qu'il \'etait n\'ecessaire d'exclure le cas des quadriques par la remarque \ref{remautofini}.
On peut aussi proposer un contre-exemple explicite \`a (iii) : on prend $R=k[[t]]$, $Z=Z'$ d\'efinis par l'\'equation
$X_0X_N+Q(X_1,\ldots,X_{N-1})=0$ o\`u $Q$ est une forme quadratique ordinaire sur $k$, et on choisit pour $f_\eta$ 
l'automorphisme de $\mathbb{P}^N_{\eta}$ donn\'e par l'\'equation 
$f_{\eta}([x_0:\ldots:x_N])=[t^{-1}x_0:x_1:\ldots:x_{N-1}:tx_N]$. 

\vspace{1em}

Enfin, plusieurs cas du th\'eor\`eme \ref{principal} sont d\'ej\`a connus.

Le cas $c=1$ et $d_1\geq 3$ sous sa forme (ii) est cons\'equence de \cite{GIT} Corollary 2.5 appliqu\'e \`a \cite{GIT} Proposition 4.2. 
Dans cette r\'ef\'erence, la base est un
corps de caract\'eristique nulle, mais ces arguments fonctionnent sur une base quelconque (par exemple $\Spec(\mathbb{Z})$) comme expliqu\'e dans \cite{Seshadri}.

Le cas $d_1+\ldots+d_c\geq N+1$ sous sa forme (iii) est cons\'equence du th\'eor\`eme \ref{MM} de Matsusaka et Mumford.
En effet, le fibr\'e canonique des intersections compl\`etes 
\'etant $\mathcal{O}(d_1+\ldots+d_c-N-1)$, il a des sections globales sous cette hypoth\`eses ; cela emp\^eche les intersections compl\`etes
consid\'er\'ees d'\^etre birationnellement r\'egl\'ees.

\subsubsection{Conjecture de Pukhlikov}

Dans \cite{Pukhli}, Pukhlikov, motiv\'e par des questions de g\'eom\'etrie birationnelle, consid\`ere l'\'enonc\'e ci-dessous plus g\'en\'eral que (iii) :
\begin{enumerate}[(i)]\setcounter{enumi}{3}
 \item Pour tout trait $T$ \`a corps r\'esiduel alg\'ebriquement clos, si $Z,Z'\subset \mathbb{P}^N_T$ sont des sous-$T$-sch\'emas ferm\'es r\'eguliers et
plats sur $T$
dont les fibres g\'eom\'etriques sont des intersections compl\`etes, tout automorphisme
$f_\eta : \mathbb{P}^N_{\eta} \to \mathbb{P}^N_{\eta}$ tel que $f_\eta(Z_{\eta})=Z'_{\eta}$ se prolonge en un 
automorphisme $f:\mathbb{P}^N_T\to\mathbb{P}^N_T$ tel que $f(Z)=Z'$.
\end{enumerate}

Pour \^etre pr\'ecis, Pukhlikov travaille avec $n\geq 2$ et $R=\mathbb{C}[[t]]$. 
Il montre (iv) si $c=1$ et $d_1\geq 3$, et il conjecture que,
\`a $c$ fix\'e, si $N$ est grand et quitte \`a exclure un nombre fini de multidegr\'es, (iv) est vrai.

Le cas $c\geq 2$, m\^eme dans le cas abord\'e ici o\`u $Z$ et $Z'$ sont lisses sur $T$, fait appara\^itre la difficult\'e 
suivante par rapport au cas $c=1$ trait\'e par Pukhlikov : il n'y a aucune raison pour qu'on puisse trouver $c$ 
\'equations $F_1,\ldots,F_c$ de
$Z_\eta$ se sp\'ecialisant en $c$ \'equations d\'efinissant $Z_s$ telles que $f_\eta^{-1*}F_1,\ldots,f_\eta^{-1*}F_c$ 
se sp\'ecialisent en $c$ \'equations d\'efinissant $Z'_s$. Cela emp\^eche de montrer (iii) en comparant directement les \'equations
de $Z_s$ et $Z'_s$.


\subsubsection{Plan de la preuve}

  Cette partie est organis\'ee comme suit. Dans le deuxi\`eme paragraphe, on regroupe tous les r\'esultats standards sur les intersections compl\`etes qui seront utiles
par la suite. Les deux paragraphes suivants sont consacr\'es \`a une preuve par l'absurde du th\'eor\`eme \ref{principal} sous sa forme (iii) ci-dessus.
Dans le troisi\`eme paragraphe, le lemme \ref{inclusion} fournit, \`a l'aide d'arguments g\'eom\'e\-triques, des restrictions a priori sur $Z_s$ et $Z'_s$.
Dans le quatri\`eme paragraphe, on conclut en manipulant de mani\`ere explicite les \'equations de petit degr\'e de $Z_s$ et $Z'_s$. On distingue pour cela 
trois cas : le cas $d_1\geq 3$ est trait\'e au paragraphe \ref{paragtrois}, le cas $c\geq 2$, $d_1=d_2=2$ au paragraphe \ref{paragdeux}, et 
le cas $c\geq 2$, $d_1=2$, $d_2\geq 3$, plus d\'elicat, au paragraphe \ref{paragfin}.

\subsection{G\'en\'eralit\'es sur les intersections compl\`etes}

On rassemble ici des r\'esultats standards sur les intersections compl\`etes pour pouvoir y faire r\'ef\'erence par la suite. Dans ce paragraphe, 
et dans ce paragraphe seulement, on autorise $c=0$.

Si $Z$ est une intersection compl\`ete, une \og suite r\'eguli\`ere globale\fg{} d\'efinissant $Z$ est la donn\'ee de $c$ \'equations homog\`enes la
d\'efinissant.

\begin{prop}\label{propic}
 
Soit $k$ un corps alg\'ebriquement clos et $Z\subset\mathbb{P}^N_k$
une intersection compl\`ete d\'efinie par des \'equations $F_1,\ldots,F_c$. On note $\mathcal{I}_Z$ son faisceau
d'id\'eaux.

\begin{enumerate}[(i)]
 \item Les \'equations $F_1,\ldots,F_c$ sont une suite r\'eguli\`ere ; $Z$ est Cohen-Macaulay.

\item Si $0<q<n$ et $d\in\mathbb{Z}$, $H^q(Z,\mathcal{O}_Z(d))=0$.
\item Si $d\in\mathbb{Z}$, la fl\`eche de restriction $H^0(\mathbb{P}^N_k,\mathcal{O}(d))\to H^0(Z,\mathcal{O}_Z(d))$ est surjective. 
Son noyau, \'egal \`a $H^0(\mathbb{P}^N_k,\mathcal{I}_Z(d))$, est constitu\'e
des polyn\^omes homog\`enes de la forme $\sum_i Q_i F_i$, $Q_i\in H^0(\mathbb{P}^N_k,\mathcal{O}(d-d_i))$. La
dimension de ce noyau ne d\'epend que de $N$, $c$, $d_1,\ldots,d_c$ et $d$.
\item La sch\'ema $Z$ est connexe. Si $Z$ est lisse, $Z$ est int\`egre.
\item Le sous-sch\'ema $Z$ n'est sch\'ematiquement inclus dans aucun hyperplan de $\mathbb{P}^N_k$.
\item Toute base de $H^0(\mathbb{P}^N_k,\mathcal{I}_Z(d_1))$ peut \^etre compl\'et\'ee en une suite r\'eguli\`ere globale d\'efinissant $Z$.
\item Si $F\in H^0(\mathbb{P}^N_k,\mathcal{I}_Z(d))$ ne s'\'ecrit pas $\sum_i Q_i F_i$ o\`u la somme porte sur les $i$ tels que $d_i<d$, 
$F$ fait partie d'une suite r\'eguli\`ere globale d\'efinissant $Z$.
\item Si $Z$ est lisse, les vari\'et\'es $\{F_i=0\}$ sont lisses et transverses en tout point de $Z$.

\end{enumerate}

\end{prop}

\begin{proof}[$\mathbf{Preuve}$]

\begin{enumerate}[(i)]
 \item 

Posons $Z_i=\{F_1=\ldots=F_i=0\}$. Par Hauptidealsatz, $Z_{i+1}$ est de codimension au plus $1$ dans $Z_i$, de sorte que, comme $Z=Z_c$ est de codimension $c$ dans 
$\mathbb{P}^N_k=Z_0$, $Z_{i}$ est n\'ecessairement de codimension pure $i$ dans $\mathbb{P}^N_k$. Montrons alors par r\'ecurrence sur $0\leq i\leq c$ que
$F_1,\ldots,F_i$ forment une suite r\'eguli\`ere et que $Z_i$ est Cohen-Macaulay. Pour $i=0$, c'est \'evident. Supposons-le vrai pour $i$. Alors $F_{i+1}$ 
ne s'annule pas sur une composante irr\'eductible de $Z_i$ car $Z_{i+1}$ est de codimension $1$ dans $Z_i$. Il ne s'annule pas non plus sur un point immerg\'e de $Z_i$
car $Z_i$, Cohen-Macaulay, n'en a pas. Ainsi, $F_{i+1}$ n'est pas un diviseur de z\'ero sur $Z_i$ de sorte que $F_1,\ldots,F_{i+1}$ forment une suite r\'eguli\`ere.
On en d\'eduit que $Z_{i+1}$ est localement intersection compl\`ete, donc Cohen-Macaulay. Cela conclut la r\'ecurrence ; on obtient l'\'enonc\'e voulu
en faisant $i=c$.

\item On montre ceci par r\'ecurrence sur $c$, le cas $c=0$ \'etant connu. Comme, par (i), $F_c$ n'est pas un diviseur de z\'ero sur $Z_{c-1}$, la
multiplication par $F_c$ induit une suite exacte courte de 
faisceaux sur $Z_{c-1}$ : $0\to \mathcal{O}_{Z_{c-1}}(d-d_c)\to \mathcal{O}_{Z_{c-1}}(d)\to\mathcal{O}_{Z_{c}}(d)\to 0$. En \'ecrivant la suite exacte longue
de cohomologie associ\'ee et en utilisant les annulations donn\'ees par l'hypoth\`ese de r\'ecurrence appliqu\'ee \`a $Z_{c-1}$, on obtient les annulations
d\'esir\'ees.

\item

Que le noyau de cette fl\`eche de restriction soit $H^0(\mathbb{P}^N_k,\mathcal{I}_Z(d))$ r\'esulte de la suite longue de cohomologie associ\'ee
\`a la suite exacte courte $0\to\mathcal{I}_Z\to\mathcal{O}_{\mathbb{P}^N_k}\to\mathcal{O}_Z\to 0$.

Montrons la surjectivit\'e de l'application de restriction. On raisonne par r\'ecurrence sur $c$ et on utilise la suite exacte longue de cohomologie suivante :
$0\to H^0(Z_{c-1},\mathcal{O}(d-d_c))\to H^0(Z_{c-1},\mathcal{O}(d))\to\ H^0(Z_{c},\mathcal{O}(d))\to 
H^1(Z_{c-1},\mathcal{O}(d-d_c))$.
Par (ii) le $H^1$ s'annule, ce qui permet de conclure.

En particulier, on a une suite exacte courte : $0\to H^0(\mathbb{P}^N_k,\mathcal{I}_Z(d))\to H^0(\mathbb{P}^N_k,\mathcal{O}_{\mathbb{P}^N_k}(d))\to
 H^0(Z,\mathcal{O}_Z(d))\to 0$. Pour montrer que $h^0(\mathbb{P}^N_k,\mathcal{I}_Z(d))$ ne d\'epend pas que de $N$, $c$, $d_1,\ldots,d_c$ et $d$,
il suffit de montrer que c'est le cas de $h^0(Z,\mathcal{O}_Z(d))$. Cela se montre par r\'ecurrence sur $c$ en utilisant \`a nouveau
la suite exacte 
$0\to H^0(Z_{c-1},\mathcal{O}(d-d_c))\to H^0(Z_{c-1},\mathcal{O}(d))\to\ H^0(Z_{c},\mathcal{O}(d))\to 0$.

Il reste \`a d\'ecrire les polyn\^omes homog\`enes inclus dans ce noyau. On raisonne encore par r\'ecurrence sur $c$, et on proc\`ede par chasse au diagramme
dans le diagramme commutatif 
ci-dessous o\`u l'on a vu que la deuxi\`eme ligne est exacte, o\`u les fl\`eches verticales sont surjectives de noyau connu par hypoth\`ese de 
r\'ecurrence et o\`u les fl\`eches horizontales de gauche sont donn\'ees
par la multiplication par $F_c$.

\begin{equation*}
\xymatrix @C=2mm @R=3mm {
&H^0(\mathbb{P}^N_k,\mathcal{O}(d-d_c))\ar[r]\ar[d]
&H^0(\mathbb{P}^N_k,\mathcal{O}(d))\ar[d]&        \\
0\ar[r]& H^0(Z_{c-1},\mathcal{O}(d-d_c))\ar[r]&H^0(Z_{c-1},\mathcal{O}(d))\ar[r]
&H^0(Z_{c},\mathcal{O}(d)) \ar[r] & 0                 
}
\end{equation*}

\item On applique (iii) avec $d=0$. Il vient $H^0(Z,\mathcal{O}_Z)=k$ : $Z$ est bien connexe. Si $Z$ est lisse et connexe, elle est int\`egre.
\item On applique (iii) avec $d=1$ : la description explicite de $H^0(\mathbb{P}^N_k,\mathcal{I}_Z(1))$ montre qu'il est nul de sorte que 
$H^0(\mathbb{P}^N_k,\mathcal{O}(1))\to H^0(Z,\mathcal{O}_Z(1))$ est injective, ce qu'on voulait.
\item On applique (iii) avec $d=d_1$ : $H^0(\mathbb{P}^N_k,\mathcal{I}_Z(d_1))$ est constitu\'e des combinaisons lin\'eaires \`a coefficients dans $k$ des 
$F_i$ qui sont de degr\'e $d_1$. Si l'on remplace
les $F_i$ de degr\'e $d_1$ par une autre base de $H^0(\mathbb{P}^N_k,\mathcal{I}_Z(d_1))$, on obtient une nouvelle collection de $c$ \'equations d\'efinissant $Z$.
\item Par (iii), on peut \'ecrire $F=\sum_i Q_iF_i$ et il existe $i$ tel que $d=d_i$ et $Q_i$ est un scalaire non nul. Alors, en rempla\c{c}ant $F_i$ par $F$, 
on obtient on obtient une nouvelle collection de $c$ \'equations d\'efinissant $Z$.
\item Par lissit\'e, si $z\in Z$, $T_z Z$ est de codimension $c$ dans $T_z\mathbb{P}^N_k$. Comme 
$T_z Z=\cap_{i=1}^c T_z\{F_i=0\}$, les $T_z\{F_i=0\}$ sont n\'ecessairement
des hyperplans transverses de $T_z\mathbb{P}^N_k$, de sorte que les $\{F_i=0\}$ sont lisses et transverses en $z$.
\end{enumerate}
\end{proof}

\begin{lemme}\label{cdb}
Soit $T$ un trait \`a corps r\'esiduel alg\'ebriquement clos et
$Z\subset \mathbb{P}^N_T$ un sous-$T$-sch\'ema ferm\'e
plat sur $T$ dont les fibres g\'eom\'etriques sont des intersections compl\`etes.
Soit $F\in H^0(\mathbb{P}^N_k,\mathcal{O}(d))$ un polyn\^ome homog\`ene de degr\'e $d$ \`a
coefficients dans $k$ s'annulant 
sur $Z_s$. Alors on peut relever $F$ en un polyn\^ome homog\`ene de degr\'e $d$ \`a
coefficients dans $R$ s'annulant 
sur $Z$.
\end{lemme}

\begin{proof}[$\mathbf{Preuve}$]

Notons $\mathcal{I}_Z$ le faisceau d'id\'eaux de $Z$. Comme $\mathbb{P}^N_T$ et $Z$ sont plats sur $T$, la suite exacte 
$0\to\mathcal{I}_Z\to\mathcal{O}_{\mathbb{P}^N_T}\to\mathcal{O}_Z\to 0$ montre que $\mathcal{I}_Z$ est plat sur $T$. De plus, cette suite exacte 
reste alors exacte apr\`es restriction \`a la fibre sp\'eciale, de sorte que $\mathcal{I}_Z|_{\mathbb{P}^N_k}=\mathcal{I}_{Z_s}$.

 On a $h^0(\mathbb{P}^N_\eta,\mathcal{I}_Z(d))=h^0(\mathbb{P}^N_{\bar{K}},\mathcal{I}_Z(d))=h^0(\mathbb{P}^N_{k},\mathcal{I}_{Z_s}(d))$ o\`u
la premi\`ere \'egalit\'e d\'ecoule de \cite{Hartshorne} III 9.3 et la seconde de \ref{propic} (iii). Ainsi, les hypoth\`eses de \cite{Hartshorne} III 12.9
appliqu\'e au faisceau $\mathcal{I}_Z(d)$ sur $\mathbb{P}^N_T$ et au morphisme structurel $q:\mathbb{P}^N_T\to T$ sont v\'erifi\'ees.
Par cons\'equent,
$q_*\mathcal{I}_Z(d)$ est localement libre sur $T$ (donc libre car $T$ est local), et 
$q_*\mathcal{I}_Z(d)\otimes_R k\to H^0(\mathbb{P}^N_k,\mathcal{I}_{Z_s}(d))$ est 
un isomorphisme.

On en d\'eduit que $F\in H^0(\mathbb{P}^N_k,\mathcal{I}_{Z_s}(d))$ se rel\`eve en un \'el\'ement de $\tilde{F}\in H^0(\mathbb{P}^N_T,\mathcal{I}_Z(d))$, 
c'est-\`a dire en un \'el\'ement de 
$H^0(\mathbb{P}^N_T,\mathcal{O}(d))$ nul sur $Z$. On conclut car, par
\cite{Hartshorne} III 5.1 (a), $H^0(\mathbb{P}^N_T,\mathcal{O}(d))$ est constitu\'e des polyn\^omes homog\`enes de degr\'e $d$ \`a
coefficients dans $R$.
\end{proof}





 






\subsection{Automorphismes projectifs}

Commen\c{c}ons la preuve du th\'eor\`eme \ref{principal}. On fixe pour cela un trait $T$ \`a corps r\'esiduel alg\'ebriquement clos,
$Z,Z'\subset \mathbb{P}^N_T$ des sous-$T$-sch\'emas ferm\'es
plats sur $T$ dont les fibres g\'eom\'etriques sont des intersections compl\`etes lisses, 
et un automorphisme
$f_\eta : \mathbb{P}^N_{\eta} \to \mathbb{P}^N_{\eta}$ tel que $f_\eta(Z_{\eta})=Z'_{\eta}$.

\subsubsection{Description de l'automorphisme $f_\eta$}

\begin{lemme}\label{ba}
On peut supposer qu'il existe des entiers $\alpha_0\leq\ldots\leq\alpha_N$ tels que $f_\eta$ soit donn\'e par la formule
$f_{\eta}([x_0:\ldots:x_N])=[t^{\alpha_0}x_0:\ldots:t^{\alpha_N}x_N]$.
\end{lemme}

\begin{proof}[$\mathbf{Preuve}$]
L'automorphisme $f_{\eta}$ est induit par une automorphisme lin\'eaire $f_K:K^{N+1}\to K^{N+1}$. Quitte \`a composer $f_K$ avec une homoth\'etie, 
on peut supposer que $f_K$ induit
$f_R:R^{N+1}\to R^{N+1}$ $R$-lin\'eaire. Comme $f_R\otimes K$ est surjective, $\Coker(f_R)$ est de torsion, de sorte que $f_R$ est injective et
$\Ima(f_R)$ est un sous-$R$-module de rang maximal.
Par le th\'eor\`eme de la base adapt\'ee,
on peut trouver des \'el\'ements $e_0,\ldots, e_{N}$ de $R^{N+1}$, $f_0,\ldots,f_N$ de $R^{N+1}$
et $\lambda_0,\ldots,\lambda_N$ de $R$ tels que $(f_R(e_i))$ soit une base de $\Ima(f_R)$, $(f_i)$ soit une base de $R^{N+1}$ et $f_R(e_i)=\lambda_i f_i$. Comme 
tout \'el\'ement de $R$ s'\'ecrit comme une unit\'e fois une puissance de l'uniformisante, on peut supposer $\lambda_i=t^{\alpha_i}$. 
Quitte \`a r\'eordonner les $e_i$ et les $f_i$, on peut supposer que $\alpha_0\leq\ldots\leq\alpha_N$. Remarquons enfin que comme $f_R$ est injective,
les $e_i$ forment une base de $R^{N+1}$.

On a montr\'e que quitte \`a composer \`a la source et au but par un automorphisme de $\mathbb{P}^N_T$, $f_\eta$ est de la forme voulue.
\end{proof}

D\'esormais, on suppose que $f_\eta$ est donn\'e par une telle formule. 

Si $\alpha_0=\ldots=\alpha_N$, $f_\eta$ est l'identit\'e 
et se prolonge donc en l'identit\'e $f:\mathbb{P}^N_T\to\mathbb{P}^N_T$. Comme, par platitude, $Z$ et $Z'$ sont les adh\'erences de $Z_{\eta}$ et $Z'_{\eta}$, et que 
$f_\eta(Z_{\eta})=Z'_{\eta}$, on a $f(Z)=Z'$ comme voulu. 

Dans toute la suite, on suppose au contraire que $\alpha_0<\alpha_N$, et on cherche \`a obtenir une contradiction.

\subsubsection{Sp\'ecialisation de l'automorphisme $f_\eta$}

On d\'efinit $p_*$ et $p^*$ de sorte que $\alpha_0=\ldots=\alpha_{p_*}<\alpha_{p_*+1}$ et
$\alpha_N=\ldots=\alpha_{N-p^*}>\alpha_{N-p^*-1}$. On note $P_*=\{X_{p_*+1}=\ldots=X_N=t=0\}$, $P^*=\{X_0=\ldots=X_{N-p^*-1}=t=0\}$, 
$L_*:=\{X_{N-p^*}=\ldots=X_N=t=0\}$ et $L^*:=\{X_0=\ldots=X_{p_*}=t=0\}$ ;
ce sont des sous-espaces lin\'eaires
de $\mathbb{P}^N_k$.

\begin{lemme}\label{inclusion}
On a $P_*\subset Z'_{s}$. De m\^eme, $P^*\subset Z_{s}$.
\end{lemme}

\begin{proof}[$\mathbf{Preuve}$]

On montre l'\'enonc\'e concernant $Z'_{s}$ ; l'autre est sym\'etrique.

L'isomorphisme $f_\eta$ induit une application rationnelle $f:\mathbb{P}^N_T\dashrightarrow\mathbb{P}^N_T$. L'expression de $f_\eta$ montre que $f$
est d\'efinie hors de $L^*$ et que 
sa restriction $f_s$ \`a la fibre sp\'eciale est la projection depuis $L^*$ sur $P_*$. 

Notons $W$ l'adh\'erence de $f_s(Z_{s}\setminus (Z_{s}\cap L^*))$, munie de sa structure r\'eduite. Par description de $f_s$, $W\subset P_*$.
Comme, par platitude, $Z$ et $Z'$ sont les adh\'erences de $Z_{\eta}$ et $Z'_{\eta}$, et que
$f_\eta(Z_{\eta})\subset Z'_{\eta}$, on a $f(Z\setminus (Z\cap L^*))\subset Z'$. En se restreignant aux fibres sp\'eciales, il vient :
 $W\subset Z'_{s}$. Si $W=P_*$, on peut conclure ; on suppose par l'absurde que ce n'est pas le cas.

Remarquons que $Z_{s}$ n'est pas ensemblistement inclus dans $L^*$. Si c'\'etait le cas, comme $Z_{s}$ est r\'eduit, il serait sch\'ematiquement inclus
dans $L^*$, donc dans un hyperplan de $\mathbb{P}^N_k$, et cela contredit \ref{propic} (v). Ainsi, $(Z_{s}\cap L^*)\neq Z_{s}$ de sorte que,
$Z_{s}$ \'etant int\`egre par \ref{propic} (iv), 
$Z_{s}\setminus (Z_{s}\cap L^*)$ est dense dans $Z_{s}$, donc de dimension $n$. 

Si $f_s:Z_{s}\setminus (Z_{s}\cap L^*)\to W$ \'etait g\'en\'eriquement finie, $W$ serait de dimension $n$.
Comme $W\subset Z'_{s}$, $W$ serait une composante irr\'eductible de $Z'_{s}$, et comme $Z'_{s}$ est int\`egre par \ref{propic} (iv), 
on aurait $Z'_{s}=W$, donc $Z'_{s}\subset P_*$.
Alors $Z'_{s}$ serait sch\'ematiquement inclus dans un hyperplan de $\mathbb{P}^N_k$, ce qui contredit \ref{propic} (v). On a montr\'e que 
$f_s:Z_{s}\setminus (Z_{s}\cap L^*)\to W$
n'est pas g\'en\'eriquement finie.

Soit maintenant $w$ un point ferm\'e de $W$ ; on choisit $w$ g\'en\'eral de sorte que $w$ est un point lisse de $W$, et que la fibre $F$ de 
$f_s:Z_{s}\setminus (Z_{s}\cap L^*)\to W$ en $w$ est de dimension $\geq 1$. Comme $w$ est un point lisse de $W$ et que $W\neq P_*$,
il r\'esulte que $T_w W$ est un sous-espace
lin\'eaire strict de $P_*$. Notons $C$ le c\^one r\'eduit de sommet $L^*$ et de base $W$ : on a $Z_{s}\subset C$ ensemblistement car 
$f_s(Z_{s}\setminus (Z_{s}\cap L^*))\subset W$, donc 
sch\'ematiquement car $Z_{s}$ est r\'eduit. Ainsi, pour tout $f\in F$, comme $T_f C=\langle T_w W, L^*\rangle$, $T_{f}Z_{s}\subset \langle T_w W, L^*\rangle$.
Comme $T_w W$ est un sous-espace lin\'eaire strict de $P_*$, $\langle T_w W, L^*\rangle$ est un sous-espace lin\'eaire strict de $\mathbb{P}^N_k$. On obtient une 
contradiction car, par \cite{LazarsfeldII} 6.3.5, 6.3.6, une hypersurface de $\mathbb{P}^N_k$ ne peut pas \^etre tangente \`a $Z_{s}$ le long
d'une sous-vari\'et\'e de dimension $\geq 1$.
\end{proof}

\begin{cor}\label{eqlisse}
Si $F\in H^0(\mathbb{P}^N_k,\mathcal{O}(d))$ fait partie d'une suite r\'eguli\`ere globale d\'efinissant $Z_{s}$, $\{F=0\}$ contient $P^*$ et
y est lisse.

De m\^eme, si $F'\in H^0(\mathbb{P}^N_k,\mathcal{O}(d))$ fait partie d'une suite r\'eguli\`ere globale d\'efinissant $Z'_{s}$, $\{F'=0\}$ contient $P_*$ et
y est lisse.
\end{cor}

\begin{proof}[$\mathbf{Preuve}$]

C'est une cons\'equence du lemme \ref{inclusion} et de \ref{propic} (viii).
\end{proof}

\subsection{\'Etude des \'equations de petit degr\'e}

Notons $\mathfrak{M}_d$ l'ensemble des mon\^omes de degr\'e $d$ en $X_0,\ldots,X_N$. Si
$M=X_0^{e_0}\ldots X_N^{e_N}$, on note $\deg_{\alpha}(M)=\sum_i \alpha_i e_i$ : c'est le $\alpha$-degr\'e du mon\^ome $M$. 
Si $F\in H^0(\mathbb{P}^N_k,\mathcal{O}(d))$, on dit qu'un mon\^ome $M$ intervient dans $F$ si le coefficient de $M$ dans $F$ n'est pas nul. On dit qu'une variable $X_i$
intervient dans $F$ si un mon\^ome qu'elle divise intervient dans $F$.

\subsubsection{Sp\'ecialisation d'\'equations}\label{specialisation}

Soit $F\in H^0(\mathbb{P}^N_k,\mathcal{O}(d))$ une \'equation de degr\'e $d$ non nulle de $Z_{s}$. Par le lemme \ref{cdb}, on peut la relever en 
$\tilde{F}=\sum_{M\in\mathfrak{M}_d} a_M M\in H^0(\mathbb{P}^N_T,\mathcal{O}(d))$, 
une \'equation de degr\'e $d$ non nulle de $Z$. Vu l'expression de $f_{\eta}$, $\sum t^{-\deg_{\alpha}(M)}a_M M$
est une \'equation de degr\'e $d$ non nulle de $Z'_{\eta}$. Notons $r_{\tilde{F}}=\min_{M\in\mathfrak{M}_d} (v(a_M)-\deg_{\alpha}(M))$, consid\'erons
$\tilde{F'}=\sum t^{-\deg_{\alpha}(M)-r_{\tilde{F}}}a_M M\in H^0(\mathbb{P}^N_T,\mathcal{O}(d))$ et notons $F'=\pi(\tilde{F'})\in H^0(\mathbb{P}^N_k,\mathcal{O}(d))$.
Par choix de $r_{\tilde{F}}$, $F'$ est non nulle. Comme $Z'_{\eta}\subset \{\tilde{F'}=0\}$, et que $Z'$ est son adh\'erence par platitude, $Z'\subset \{\tilde{F'}=0\}$.
Prenant les fibres sp\'eciales, on obtient $Z'_{s}\subset\{F'=0\}$ : $F'$ est une \'equation de degr\'e $d$ non nulle de $Z'_{s}$.

\subsubsection{\'Equations de degr\'e $d\geq 3$}\label{paragtrois}

On garde les notations du paragraphe \ref{specialisation}.

\begin{lemme}\label{degretrois}
Supposons que $d\geq 3$ et que $F$ fasse partie d'une suite r\'eguli\`ere globale d\'efinissant $Z_{s}$. Alors $\{F'=0\}$ contient $P_*$ et y est singulier.

De plus, les mon\^omes intervenant dans $F'$ sont de $\alpha$-degr\'e $\geq \alpha_0+(d-1)\alpha_N$.
\end{lemme}

\begin{proof}[$\mathbf{Preuve}$]

 Si aucun mon\^ome $X_{i_1}\ldots X_{i_{d-1}}X_j$ avec $N-p^*\leq i_1,\ldots,i_{d-1}\leq N$ n'intervient dans $F$, 
on voit que $P^*\subset\{F=0\}$, et le crit\`ere jacobien
montre que $\{F=0\}$ est singulier le long de $P^*$. Cela contredit le corollaire \ref{eqlisse}.

 Soit donc $M$ un mon\^ome de cette forme intervenant dans $F$. Alors
$r_{\tilde{F}}\leq v(a_M)-\deg_{\alpha}(M)=-\deg_{\alpha}(M)\leq -(\alpha_0+(d-1)\alpha_N)$.

 Soit maintenant $M'$ un mon\^ome intervenant dans $F'$, de sorte que $v(a_{M'})-\deg_{\alpha}(M')-r_{\tilde{F}}=0$. 
Alors $\deg_{\alpha}(M')=v(a_{M'})-r_{\tilde{F}}\geq-r_{\tilde{F}}\geq \alpha_0+(d-1)\alpha_N$.
Comme $d\geq 3$, et que $\alpha_0<\alpha_N$, cela implique $\deg_{\alpha}(M')>(d-1)\alpha_0+\alpha_N$. En particulier, aucun mon\^ome de la forme
$X_{j_1}\ldots X_{j_{d-1}}X_i$ avec $0\leq j_1,\ldots,j_{d-1}\leq p_*$ n'intervient dans $F'$. Cela implique que $\{F'=0\}$ contient $P_*$ 
et le crit\`ere jacobien montre que $\{F'=0\}$ est singulier le long de $P_*$.
\end{proof}

On peut \`a pr\'esent montrer le th\'eor\`eme \ref{principal} si $d_1\geq 3$.

\begin{proof}[$\mathbf{Preuve \text{ }du \text{ }th\acute{e}or\grave{e}me\text{ }\ref{principal}\text{ }si\text{ }d_1\geq 3}$]~

On prend pour $F$ une \'equation de degr\'e $d_1$ de $Z_{s}$ qui fait partie d'une suite r\'eguli\`ere globale d\'efinissant $Z_{s}$. Alors, 
par le lemme \ref{degretrois}, $\{F'=0\}$ contient $P_*$ et y est singulier.

Comme $F'$ est une \'equation non nulle de degr\'e $d_1$ de $Z'_{s}$,
par \ref{propic} (vi), elle fait partie d'une suite r\'eguli\`ere globale d\'efinissant $Z'_{s}$. Cela contredit le
corollaire \ref{eqlisse}.
\end{proof}

\subsubsection{\'Equations de degr\'e $2$}\label{paragdeux}

On garde les notations du paragraphe \ref{specialisation}.

\begin{lemme}\label{degredeux}
Supposons que $d=d_1=2$.

 Alors, si $N-p^*\leq i\leq N$, la variable $X_i$ intervient dans $F$, mais seulement dans 
des mon\^omes de la forme $X_iX_j$ avec $0\leq j\leq p_*$. 

De m\^eme, si $0\leq j\leq p_*$, la variable $X_j$ intervient dans $F'$, mais seulement dans 
des mon\^omes de la forme $X_iX_j$ avec $N-p^*\leq i\leq N$.

De plus, $r_{\tilde{F}}=-\alpha_0-\alpha_N$.
\end{lemme}

\begin{proof}[$\mathbf{Preuve}$]
Par \ref{propic} (vi), $F$ fait partie d'une suite r\'eguli\`ere globale d\'efinissant $Z_{s}$. En particulier, par le corollaire \ref{eqlisse},
$\{F=0\}$ contient $P^*$ et y est lisse. Si $N-p^*\leq i\leq N$ est tel que $X_i$ n'intervient pas dans $F$, 
le crit\`ere jacobien montre que $\{F=0\}$ est singulier en
le point de $P^*$ ayant toutes ses coordonn\'ees nulles sauf la $i$-\`eme : c'est absurde.
De m\^eme, si $0\leq j\leq p_*$, la variable $X_j$ appara\^it dans $F'$.

Soient maintenant $N-p^*\leq i\leq N$, $0\leq j\leq p_*$, $M$ un mon\^ome intervenant dans $F$ divisible par $X_i$ et $M'$ un mon\^ome intervenant dans $F'$
divisible par $X_j$. On a $r_{\tilde{F}}\leq v(a_M)-\deg_{\alpha}(M)=-\deg_{\alpha}(M)\leq-\alpha_0-\alpha_N$ d'une part et 
$r_{\tilde{F}}=v(a_{M'})-\deg_\alpha(M')\geq-\deg_\alpha(M')\geq-\alpha_0-\alpha_N$ d'autre part. Ces in\'egalit\'es sont donc des \'egalit\'es. En particulier, 
$\deg_{\alpha}(M)=\deg_\alpha(M')=\alpha_0+\alpha_N$, ce qui montre les restrictions voulues sur les mon\^omes intervenant dans $F$ et $F'$, et 
$r_{\tilde{F}}=-\alpha_0-\alpha_N$.
\end{proof}

Montrons \`a pr\'esent le th\'eor\`eme \ref{principal} si $c\geq 2$ et $d_1=d_2=2$.

\begin{proof}[$\mathbf{Preuve \text{ }du \text{ }th\acute{e}or\grave{e}me\text{ }\ref{principal}\text{ }si\text{ }c\geq 2\text{ }et\text{ }d_1=d_2=2}$]~

On peut supposer, quitte \`a \'echanger $Z$ et $Z'$, que $p^*\geq p_*$. Soit $(F_t)_{t\in\mathbb{P}^1}$ un pinceau d'\'equations de degr\'e $2$ de $Z_{s}$. Comme $F_t$
fait partie d'une suite r\'eguli\`ere globale d\'efinissant $Z_{s}$ par \ref{propic} (vi), le corollaire \ref{eqlisse} montre que
$\{F_t=0\}$ contient $P^*$ et y est lisse. De plus, les restrictions sur les mon\^omes de $F_t$ obtenues dans le
lemme \ref{degredeux} montrent que si $x\in P^*$, $L^*\subset T_x\{F_t=0\}$. L'ensemble des hyperplans 
de $\mathbb{P}^N_k$ contenant $L^*$ s'identifie naturellement au dual $(P_*)^\vee$ de $P_*$, et on obtient un morphisme
$\Gamma: \mathbb{P}^1\times P^*\to (P_*)^\vee$ d\'efini par $\Gamma(t,x)=T_x\{F_t=0\}$. 

Comme $\dim(\mathbb{P}^1\times P^*)=p^*+1>p_*=\dim((P_*)^\vee)$, on peut trouver un hyperplan $H\in (P_*)^\vee$,
et $C$ une courbe irr\'eductible dans $\mathbb{P}^1\times P^*$ tels que $\Gamma(t,x)=H$ pour $(t,x)\in C$.
Si la projection $C\to P^*$ n'est pas constante, son image est une courbe, et $H$ est tangent \`a $Z_{s}$ le long de cette courbe, ce qui contredit
\cite{LazarsfeldII} 6.3.5, 6.3.6. Sinon, $C=\mathbb{P}^1\times \{x\}$ pour $x\in P^*$, et tous les $\{F_t=0\}$ ont espace tangent $H$ en $x$. En particulier,
les \'equations $F_0$ et $F_1$ 
ne sont pas transverses en $x$. Comme, par \ref{propic} (vi), elles font partie d'une suite r\'eguli\`ere globale d\'efinissant $Z_{s}$, 
cela contredit \ref{propic} (viii).
\end{proof}

\subsubsection{Fin de la preuve}\label{paragfin}

Il reste \`a prouver le th\'eor\`eme \ref{principal} si $c\geq 2$, $d_1=2$ et $d_2\geq 3$.

\begin{proof}[$\mathbf{Preuve \text{ }du \text{ }th\acute{e}or\grave{e}me\text{ }\ref{principal}\text{ }si\text{ }c\geq 2,\text{ }d_1=2\text{ }et\text{ }d_2\geq3}$]~

Soient $F\in H^0(\mathbb{P}^N_k,\mathcal{O}(2))$ et $G\in H^0(\mathbb{P}^N_k,\mathcal{O}(d_2))$ des \'equations de $Z_{s}$ faisant partie d'une suite 
r\'eguli\`ere globale la d\'efinissant. Appliquant la discussion du paragraphe \ref{specialisation} \`a $F$ et $G$, on obtient d'une part 
des \'equations $\tilde{F}$,  $\tilde{F}'$, $F'$ et un entier $r_{\tilde{F}}$, d'autre part des \'equations $\tilde{G}$, 
$\tilde{G}'$, $G'$ et un entier $r_{\tilde{G}}$. Par le lemme \ref{degredeux}, $r_{\tilde{F}}=-\alpha_0-\alpha_N$.

Par le lemme \ref{degretrois}, $G'$ est singulier le long de $P_*$, de sorte que, par le corollaire \ref{eqlisse},
$G'$ ne peut faire partie d'une suite r\'eguli\`ere globale d\'efinissant $Z'_{s}$. Par \ref{propic} (vii), cela signifie qu'il existe 
$Q\in H^0(\mathbb{P}^N_k,\mathcal{O}(d_2-2))$ 
tel que $G'=QF'$. 

Par le lemme \ref{degredeux}, la variable $X_0$ intervient dans $F'$. Par le lemme \ref{degretrois}, tous les mon\^omes intervenant dans $G'$ 
ont un $\alpha$-degr\'e $\geq\alpha_0+(d_2-1)\alpha_N$. Ces deux faits impliquent que tous les mon\^omes intervenant dans $Q$ ont $\alpha$-degr\'e $(d_2-2)\alpha_N$,
et que $G'$ fait intervenir au moins un mon\^ome de $\alpha$-degr\'e $\alpha_0+(d_2-1)\alpha_N$. En particulier, on a \'egalit\'e dans les in\'egalit\'es de la 
d\'emonstration du lemme \ref{degretrois}, ce qui montre $r_{\tilde{G}}=-\alpha_0-(d_2-1)\alpha_N$.

Soit $\tilde{Q}$ un relev\'e de $Q$ \`a $H^0(\mathbb{P}^N_T,\mathcal{O}(d_2-2))$ ne faisant intervenir que des mon\^omes de
$\alpha$-degr\'e $(d_2-2)\alpha_N$. Posons $H=G-QF$. Comme $\{F=G=0\}=\{F=H=0\}$, $F$ et $H$ font partie d'une suite r\'eguli\`ere globale d\'efinissant $Z_{s}$.
On applique la discussion du paragraphe \ref{specialisation} \`a $H$ et \`a son relev\'e $\tilde{H}=\tilde{G}-\tilde{Q}\tilde{F}$. Comme le raisonnement effectu\'e
ci-dessus pour $F$ et $G$ vaut aussi pour $F$ et $H$, on a $r_{\tilde{H}}=-\alpha_0-(d_2-1)\alpha_N$. 

Calculons $\tilde{H}'$. On note $a_M^{\tilde{F}}$ le coefficient du mon\^ome $M$ dans $\tilde{F}$, et on utilise des notations
analogues pour $\tilde{G}$, $\tilde{H}$ et $\tilde{Q}$. Comme le $\alpha$-degr\'e d'un produit de mon\^omes est la somme des $\alpha$-degr\'es de ces mon\^omes,
il vient :

\begin{alignat*}{3}
\tilde{H}' &=t^{-r_{\tilde{H}}}\Bigg(\sum_{M\in\mathfrak{M}_{d_2}}&&t^{-\deg_{\alpha}(M)}a_M^{\tilde{H}}M\Bigg)\\
           &=t^{-r_{\tilde{H}}}\Bigg( \sum_{M\in\mathfrak{M}_{d_2}}&&t^{-\deg_{\alpha}(M)}a_M^{\tilde{G}}M \\
           &  &&-\sum_{M\in\mathfrak{M}_{d_2-2}}
t^{-\deg_{\alpha}(M)}a_M^{\tilde{Q}}M
\sum_{M\in\mathfrak{M}_2}t^{-\deg_{\alpha}(M)}a_M^{\tilde{F}}M\Bigg).
\end{alignat*}

Comme $r_{\tilde{G}}=r_{\tilde{H}}=-\alpha_0-(d_2-1)\alpha_N$, $r_{\tilde{F}}=-\alpha_0-\alpha_N$, et que tous les mon\^omes
intervenant dans $\tilde{Q}$ sont de $\alpha$-degr\'e
$(d_2-2)\alpha_N$, on obtient :

\begin{alignat*}{2}
\tilde{H}' &=t^{-r_{\tilde{G}}}\Bigg(\sum_{M\in\mathfrak{M}_{d_2}}t^{-\deg_{\alpha}(M)}a_M^{\tilde{G}}M\Bigg)-
\tilde{Q}\Bigg(t^{-r_{\tilde{F}}}\sum_{M\in\mathfrak{M}_2}t^{-\deg_{\alpha}(M)}a_M^{\tilde{F}}M\Bigg)\\
           &=\tilde{G}'-\tilde{Q}\tilde{F}'.
\end{alignat*}

 Sp\'ecialisant cette 
\'equation, il vient $H'=G'-QF'=0$. C'est une contradiction car, dans la construction du paragraphe \ref{specialisation}, $r_{\tilde{H}}$ est choisi de 
sorte que $H'$ soit non nul. Cela conclut la preuve.
\end{proof}

\section{Automorphismes}

\subsection{Sch\'ema en groupes des automorphismes}\label{parauto}

On regroupe dans ce paragraphe des g\'en\'eralit\'es sur les automorphismes de vari\'et\'es, n\'ecessaires pour l'\'enonc\'e et la preuve du
th\'eor\`eme \ref{autoic}.

Soit $k$ un corps ; on note $p$ sa caract\'eristique (on peut avoir $p=0$). 
Si $Z$ et $T$ sont des $k$-sch\'emas, on notera $Z_T$ le $T$-sch\'ema $Z\times_k T$.

On adopte les conventions de \cite{FGA} Chap. 9 en ce qui concerne les foncteurs et sch\'emas de Picard.

\subsubsection{Automorphismes d'une vari\'et\'e}

On rappelle que, si $Z$ est un $k$-sch\'ema projectif, le foncteur qui \`a un $k$-sch\'ema $T$ associe l'ensemble
$\Aut_T(Z_T)$ des $T$-automorphismes de $Z_T$ est repr\'esentable par un $k$-sch\'ema en groupes not\'e $\Aut_k(Z)$. 
Le groupe des composantes connexes de $\Aut_k(Z)$ est 
d\'enombrable et sa composante neutre est un $k$-sch\'ema en groupes de type fini dont l'espace tangent en l'identit\'e s'identifie \`a $H^0(Z,T_Z)$.
Pour ces faits, on pourra consulter \cite{Sernesi} Prop. 4.6.10.

\subsubsection{Automorphismes d'une vari\'et\'e polaris\'ee}

 Soient $g:Z\to\Spec(k)$ un $k$-sch\'ema projectif 
g\'eom\'etriquement int\`egre, et
$\lambda\in\mathbf{Pic}_{Z/k}(k)$. Consid\'erons le foncteur qui \`a un $k$-sch\'ema $T$ associe l'ensemble des 
$f\in\Aut_T(Z_T)$ tels que le diagramme suivant commute :

\begin{equation}\label{diagpic}
\xymatrix {
 \mathbf{Pic}_{Z_T/T}\ar^{f^*}[rr]& &\mathbf{Pic}_{Z_T/T}   \\
&    T  \ar^{\lambda_T}[ul]\ar_{\lambda_T}[ur]    &
}
\end{equation}

Ce foncteur est repr\'esentable par un sous-sch\'ema ferm\'e de $\Aut_k(Z)$. En effet, si $T$ est un $k$-sch\'ema et $f\in\Aut_T(Z\times_k T)$, 
le sous-sch\'ema ferm\'e de $T$ d\'efini par l'\'equation $f^*\circ\lambda_T-\lambda_T=0$ v\'erifie la propri\'et\'e universelle requise. 
Il est de plus imm\'ediat que c'est un sous-sch\'ema en groupes. On le note $\Aut_k(Z,\lambda)$.


\subsubsection{Cas des intersections compl\`etes}\label{autoproj}

Supposons maintenant que $Z\subset\mathbb{P}^N_k$ est une intersection compl\`ete lisse polaris\'ee par $\mathcal{O}(1)$. 
On note $i:Z\subset\mathbb{P}^N_k$ l'inclusion,
et $g$ et $h$ les morphismes structurels de $Z$ et $\mathbb{P}^N_k$.  On va 
donner une description plus concr\`ete de 
$\Aut_k(Z,\mathcal{O}(1))$, en l'identifiant \`a un sous-sch\'ema en groupes de $PGL_{N+1,k}$. 

Rappelons que $PGL_{N+1,k}=\Aut_k(\mathbb{P}^N_k)$, de sorte que si $T$ est un $k$-sch\'ema, $PGL_{N+1,k}(T)$ est l'ensemble des
$T$-automorphismes de $\mathbb{P}^N_T$ (voir \cite{GIT} 0.5b). 
Consid\'erons le sous-foncteur $G$ de $PGL_{N+1,k}$ qui associe \`a un $k$-sch\'ema $T$
l'ensemble $G(T)$ des $f\in PGL_{N+1,k}(T)$ tels que $f(Z_T)=Z_T$. V\'erifions que $G$ est repr\'esentable par un sous-sch\'ema ferm\'e de $PGL_{N+1,k}$. Pour
cela, soit $T$ un $k$-sch\'ema et $f\in PGL_{N+1,k}(T)$. Les sous-sch\'emas $Z_T$ et $f(Z_T)$ de $\mathbb{P}^N_T$ sont plats sur $T$ et induisent
donc des sections $T\to \Hilb(\mathbb{P}^N_T/T)$ qui sont des immersions ferm\'ees par \cite{EGA1} 5.4.6. L'intersection de ces deux sous-sch\'emas ferm\'es de 
$\Hilb(\mathbb{P}^N_T/T)$ s'identifie \`a un sous-sch\'ema ferm\'e de $T$ qui v\'erifie la propri\'et\'e universelle requise. On a 
montr\'e que $G$ est repr\'esentable par un sous-sch\'ema ferm\'e de $PGL_{N+1,k}$ ; que ce soit un sous-sch\'ema en groupes est imm\'ediat.

On va montrer que $G$ et $\Aut_k(Z,\mathcal{O}(1))$ co\"incident. Si $f\in G(T)$, on note $\Phi(T)(f)=f|_{Z_T}$. Comme $f$ pr\'eserve n\'ecessairement la polarisation
$\mathcal{O}(1)\in \mathbf{Pic}_{\mathbb{P}^N_T/T}(T)$ de $\mathbb{P}^N_T$, elle pr\'eserve sa restriction \`a $Z_T$, de sorte que 
$\Phi(T)(f)\in\Aut_k(Z,\mathcal{O}(1))(T)$.
On a ainsi construit un morphisme de foncteurs $\Phi: G\to \Aut_k(Z,\mathcal{O}(1))$ ; on v\'erifie ais\'ement qu'il respecte les lois de groupes. 
Montrons en deux temps que c'est un isomorphisme.

Tout d'abord, montrons que $\Phi(T)$ est injectif. Pour cela, soit $f\in \Ker(\Phi(T))$ : $f|_{Z_T}:Z_T\to Z_T$ est l'identit\'e. 
Par la proposition \ref{propic} (iii), $i^*:H^0(\mathbb{P}^N_k,\mathcal{O}(1))\to
H^0(Z,\mathcal{O}(1))$ est un isomorphisme de sorte que par changement de base plat par $T\to \Spec(k)$, $i_T^*:h_{T*}\mathcal{O}(1)\to g_{T*}\mathcal{O}(1)$
est un isomorphisme. La commutativit\'e du diagramme ci-dessous :
$$\xymatrix {
 g_{T*}\mathcal{O}(1)\ar^{f|_{Z_T}^*}[r]& g_{T*}\mathcal{O}(1)  \\
   h_{T*}\mathcal{O}(1)  \ar^{f^*}[r]\ar^{i_T^*}[u]  &h_{T*}\mathcal{O}(1)\ar^{i_T^*}[u]
}$$
montre que $f^*:h_{T*}\mathcal{O}(1)\to h_{T*}\mathcal{O}(1)$ est l'identit\'e, de sorte que $f$ est l'identit\'e.

Montrons enfin que $\Phi(T)$ est surjectif. Pour cela, soit $f:Z_T\to Z_T$ un \'el\'ement de $\Aut_k(Z,\mathcal{O}(1))(T)$ :
$f^*\mathcal{O}(1)\otimes\mathcal{O}(-1)$ est trivial dans $\mathbf{Pic}_{Z_T/T}(T)=\mathbf{Pic}_{Z/k}(T)$, donc, par
\cite{FGA} Th. 9.2.5 1, dans $\Pic_{Z/k}(T)$.
Il existe donc $\mathcal{L}\in\Pic(T)$
tel que $f^*\mathcal{O}(1)\simeq\mathcal{O}(1)\otimes g_T^*\mathcal{L}$. Par cons\'equent, $f^*$ induit un isomorphisme entre 
$g_{T*}\mathcal{O}(1)$ et $g_{T*}\mathcal{O}(1)\otimes\mathcal{L}$. Composant avec $i_T^*$, on obtient un isomorphisme
entre $h_{T*}\mathcal{O}(1)$ et $h_{T*}\mathcal{O}(1)\otimes\mathcal{L}$, donc entre les fibr\'es projectifs associ\'es : c'est un
isomorphisme $\mathbb{P}^N_T\to\mathbb{P}^N_T$. Par construction, c'est un \'el\'ement de $G(T)$ qui est un ant\'ec\'edent de $f$ par $\Phi(T)$.

\vspace{1em}

Ainsi, suivant la situation, on pourra consid\'erer $\Aut_k(Z,\mathcal{O}(1))$ comme un sous-sch\'ema en groupes ferm\'e
de $\Aut_k(Z)$ ou de $PGL_{N+1,k}$.

\subsection{$\mkern-5mu$Automorphismes des intersections compl\`etes lisses}

\subsubsection{\'Enonc\'e du th\'eor\`eme}

Le but de ce chapitre est de montrer que, si $Z$ est une intersection compl\`ete lisse, $\Aut_k(Z)$ et $\Aut_k(Z,\mathcal{O}(1))$ sont, 
sauf pour un petit nombre d'exceptions qu'on explique, des sch\'emas en groupes finis r\'eduits.

Le th\'eor\`eme principal est le suivant :

\begin{thm}\label{autoic}
Soit $Z$ une intersection compl\`ete lisse. 
Les sch\'emas en groupes $\Aut_k(Z)$ et $\Aut_k(Z,\mathcal{O}(1))$ co\"incident et sont finis r\'eduits, sauf dans les cas suivants :

\begin{enumerate}[(i)]
 \item \textbf{Quadriques} : si $c=1$ et $d_1=2$, $\Aut_k(Z)=\Aut_k(Z,\mathcal{O}(1))$ est lisse de dimension $\frac{N(N+1)}{2}$.

\item \textbf{Courbes de genre} $\mathbf{1}$ : si $N=2$, $c=1$ et $d_1=3$ ou si $N=3$, $c=2$ et $d_1=d_2=2$, $\Aut_k(Z)$ est de type fini et sa composante neutre
est une courbe elliptique. De plus $\Aut_k(Z,\mathcal{O}(1))$ est fini ; il est aussi r\'eduit sauf
si $N=2$, $c=1$, $d_1=3$ et $p=3$ ou si $N=3$, $c=2$, $d_1=d_2=2$ et $p=2$. 

\item \textbf{Courbes de genre} $\mathbf{\geq 2}$ : dans les autres cas o\`u $n=1$, $\Aut_k(Z)$ et $\Aut_k(Z,\mathcal{O}(1))$ sont tous deux finis r\'eduits, 
mais peuvent
ne pas co\"incider.

\item \textbf{Surfaces K3} : si $N=3$, $c=1$ et $d_1=4$, si $N=4$, $c=2$, $d_1=2$ et $d_2=3$ ou si $N=5$, $c=3$ et $d_1=d_2=d_3=2$, $\Aut_k(Z)$ est de dimension nulle,
r\'eduit et au plus d\'enombrable tandis que $\Aut_k(Z,\mathcal{O}(1))$ est fini r\'eduit.

\item \textbf{Intersections de deux quadriques} : si $N\geq 5$ est impair, $c=2$, $d_1=d_2=2$ et $p=2$, $\Aut_k(Z)=\Aut_k(Z,\mathcal{O}(1))$ est fini non r\'eduit.

\end{enumerate}

\end{thm}

\subsubsection{Cas des hypersurfaces}\label{cashyp}

Le cas des hypersurfaces est classique. En caract\'eristique nulle, il est trait\'e par Kodaira et Spencer dans \cite{KS} 14.2. En caract\'eristique positive, 
la finitude des groupes d'automorphismes 
est \'etudi\'ee dans \cite{MMonsky}. On trouvera une discussion tr\`es d\'etaill\'ee incluant les probl\`emes de r\'eduction dans \cite{KatzSarnak}
11.5, 11.6, 11.7.

\subsubsection{Codimension sup\'erieure}

Les arguments en codimension sup\'erieure sont analogues.
D'une part, il faut \'etendre aux intersections compl\`etes lisses
le calcul fait dans \cite{KS} des champs de vecteurs sur une hypersurface lisse. C'est l'objet du paragraphe \ref{cdv}.
D'autre part, un ph\'enom\`ene nouveau appara\^it : la non r\'eduction de $\Aut_k(Z)$ 
quand $N\geq 5$ est impair, $c=2$, $d_1=d_2=2$ et $p=2$ (cas (v) ci-dessus). On traite ce cas \`a l'aide de calculs explicites
au paragraphe \ref{partieicquadriques}. La preuve proprement dite du th\'eor\`eme \ref{autoic} se trouve au paragraphe \ref{preuveautoic}.

Dans le cas (ii) ci-dessus, les automorphismes infinit\'esimaux sont facilement explicables : ce sont 
des automorphismes de translation de la courbe de genre $1$. Il serait int\'eressant d'obtenir \'egalement dans le cas (v)
une description g\'eom\'etrique de ces automorphismes infinit\'esimaux. Pourrait-on m\^eme identifier la composante connexe de l'identit\'e de 
$\Aut_k(Z)$ ?

\subsection{Preuve du th\'eor\`eme}

\subsubsection{Champs de vecteurs sur les intersections com\-pl\`etes lisses}\label{cdv}

Soit $Z$ une intersection compl\`ete lisse sur $k$. L'objectif de ce paragraphe est la proposition \ref{champdev} : on montre que, 
\`a quelques exceptions \'eventuelles pr\`es, $Z$
n'admet pas de champs de vecteurs globaux non triviaux. La preuve, qui \'etend celle de \cite{KS} pour les hypersurfaces,
consiste en un calcul de cohomologie de faisceaux de formes diff\'erentielles.

Rappelons tout d'abord le th\'eor\`eme d'annulation suivant sur $\mathbb{P}^N$. 
En caract\'eristique $0$, c'est une cons\'equence du th\'eor\`eme d'annulation de Bott. 
On peut en trouver une preuve de Deligne, ind\'ependante de la caract\'eristique, dans \cite{SGA7} Expos\'e XI, Th. 1.1.

\begin{lemme}\label{Bott}
On a $H^q(\mathbb{P}^N_k,\Omega_{\mathbb{P}^N_k}^p(l))=0$ sauf dans les trois cas suivants :
\begin{enumerate}[(i)]
\item $p=q$ et $l=0$,
\item $q=0$ et $l>p$,
\item $q=N$ et $l<p-N$.
\end{enumerate}
\end{lemme}

On peut en d\'eduire des th\'eor\`emes d'annulation sur $Z$.

\begin{lemme}\label{annulation}
Soit $Z$ une intersection compl\`ete lisse sur $k$.

 Si $p+q>n$ et $l>p-q$, on a $H^q(Z,\Omega_Z^p(l))=0$.
\end{lemme}

\begin{proof}[$\mathbf{Preuve}$]
On raisonne par r\'ecurrence sur $q$. Si $q=0$, on a $p>n$, de sorte que $\Omega_Z^p=0$ et l'annulation est \'evidente.

Consid\'erons la $(p+c)$-\`eme puissance ext\'erieure de la suite exacte courte $0\to N_{Z/\mathbb{P}^N_k}^*\to \Omega_{\mathbb{P}^N_k}|_Z\to\Omega_Z\to 0$.
On obtient une filtration de $\Omega^{p+c}_{\mathbb{P}^N_k}|_Z$ dont les gradu\'es successifs sont les faisceaux localement libres :
$$\Omega_Z^{p+t}\otimes\bigwedge^{c-t}N_{Z/\mathbb{P}^N_k}^*=\bigoplus_{1\leq j_1<\ldots<j_{c-t}\leq c}\Omega_Z^{p+t}(-d_{j_1}-\ldots-d_{j_{c-t}}),\text{ }t\geq 0.$$

Tensorisons ce faisceau localement libre filtr\'e par $\mathcal{O}(l+d_1+\ldots+d_c)$, de sorte que le gradu\'e correspondant \`a $t=0$ soit $\Omega^p_Z(l)$.
En d\'evissant cette filtration en suite exactes courtes de faisceaux
localement libres, et en \'ecrivant les suites exactes longues de cohomologie associ\'ees \`a ces suites
exactes courtes, on est ramen\'es, pour montrer l'annulation d\'esir\'ee, \`a v\'erifier les annulations suivantes :
\begin{enumerate}[(i)]
 \item Pour $t\geq 1$ et $1\leq j_1<\ldots<j_{t}\leq c$, $H^{q-1}(Z, \Omega_Z^{p+t}(d_{j_1}+\ldots+d_{j_{t}}+l))=0$.
 \item  $H^q(Z,\Omega^{p+c}_{\mathbb{P}^N_k}|_Z(l+d_1+\ldots+d_c))=0$.
\end{enumerate}

Pour les premi\`eres annulations, on peut appliquer l'hypoth\`ese de r\'ecurrence. Les hypoth\`eses sont v\'erifi\'ees car $p+t+q-1\geq p+q>n$ et 
$d_{j_1}+\ldots+d_{j_{t}}+l\geq l+t+1>p+t-q+1$.

Pour la seconde annulation, on dispose de la r\'esolution de Koszul de $\mathcal{O}_Z$ :
$$0\to \mathcal{K}^{-c}\to \ldots\to\mathcal{K}^0\to \mathcal{O}_Z\to 0,$$
o\`u $\mathcal{K}^{-r}=\bigwedge^r(\oplus_{i=1}^c\mathcal{O}(-d_i))=\bigoplus_{1\leq j_1<\ldots<j_{r}\leq c}\mathcal{O}(-d_{j_1}-\ldots-d_{j_r})$.
Tensorisons cette r\'esolution par le faisceau localement libre $\Omega^{p+c}_{\mathbb{P}^N_k}(l+d_1+\ldots+d_c)$, et consid\'erons la suite spectrale 
d'hypercohomologie associ\'ee :
$$ E_1^{r,s}=H^s(\mathbb{P}^N_k,\mathcal{K}^r\otimes\Omega^{p+c}_{\mathbb{P}^N_k}(l+d_1+\ldots+d_c))
\Rightarrow H^{r+s}(Z,\Omega^{p+c}_{\mathbb{P}^N_k}|_Z(l+d_1+\ldots+d_c)).$$
Il s'ensuit que pour montrer l'annulation voulue, il suffit de v\'erifier l'annulation des 
$H^{q+r}(\mathbb{P}^N_k,\Omega^{p+c}_{\mathbb{P}^N_k}(l+d_{j_1}+\ldots+d_{j_{c-r}}))$ pour $0\leq r\leq c$ et $1\leq j_1<\ldots<j_{c-r}\leq c$. Pour cela, montrons
que le lemme \ref{Bott} s'applique.

Si l'on \'etait cans le cas (i), on pourrait \'ecrire $l=-d_{j_1}-\ldots-d_{j_{c-r}}\leq -2(c-r)=(p-q)-(c-r)\leq p-q$, ce qui est absurde. 
Si l'on \'etait dans le cas (ii),
on aurait $q+r=0$ donc $q=0$, mais ce cas a \'et\'e trait\'e comme initialisation de la r\'ecurrence. 
Enfin, si l'on \'etait dans le cas (iii), il viendrait : $p-q<l\leq l+d_{j_1}+\ldots+d_{j_{c-r}}<p+c-N=p-n$, de sorte que $q>n$ et que l'annulation de
$H^q(Z,\Omega_Z^p(l))$ \'etait \'evidente.
\end{proof}

La strat\'egie de d\'emonstration du lemme ci-dessus permet de montrer l'annulation d'autres groupes de cohomologie. 
Les deux lemmes qui suivent en sont des exemples, dont on aura besoin. Le premier de ces lemmes concerne 
les hypersurfaces cubiques de dimension au moins $2$.

\begin{lemme}\label{annulationbis}
Supposons que $c=1$, $d_1=3$ et $N\geq 3$. 
Soit $Z$ une intersection compl\`ete lisse sur $k$.

Alors $H^{N-1}(Z,\Omega^1_{Z}(2-N))=0$.
\end{lemme}

\begin{proof}[$\mathbf{Preuve}$]
 On proc\`ede de la m\^eme mani\`ere que dans la preuve du lemme
\ref{annulation}. Par l'argument de filtration, on est ramen\'es \`a montrer l'annulation des deux groupes de cohomologie 
$H^{N-2}(Z,\Omega_Z^2(5-N))$ et $H^{N-1}(Z,\Omega^2_{\mathbb{P}^N_k}|_Z(5-N))$. Pour le premier, on peut appliquer le lemme \ref{annulation}. 
Pour le second, on proc\`ede de la m\^eme mani\`ere qu'en \ref{annulation} : par l'argument de suite spectrale de Koszul, 
on est ramen\'es \`a montrer l'annulation des deux groupes de cohomologie 
$H^{N-1}(\mathbb{P}^N_k,\Omega_{\mathbb{P}^N_k}^2(5-N))$ et $H^{N}(\mathbb{P}^N_k,\Omega^2_{\mathbb{P}^N_k}(2-N))$. Le lemme \ref{Bott} montre qu'ils s'annulent, sauf
le second si $N=2$.
\end{proof}

Le second de ces lemmes concerne les intersections de deux quadriques de dimension paire.

\begin{lemme}\label{annulationter}
Supposons que $c=2$, $d_1=d_2=2$ et que $N$ est pair. 
Soit $Z$ une intersection compl\`ete lisse sur $k$.

Alors $H^{N-2}(Z,\Omega^1_{Z}(3-N))=0$.
\end{lemme}

\begin{proof}[$\mathbf{Preuve}$]
  On va en fait prouver l'\'enonc\'e plus g\'en\'eral suivant : si $c=2$, $d_1=d_2=2$ et $i\geq 0$,  $H^{N-2-i}(Z,\Omega^{1+i}_{Z}(3+2i-N))=0$,
de sorte qu'on obtient le r\'esultat voulu en faisant $i=0$.
  La preuve proc\`ede par r\'ecurrence descendante sur $i$. L'initialisation de la r\'ecurrence
est facile : si $i\geq N-2$, le faisceau $\Omega^{1+i}_Z$ est nul par dimension. Supposons l'annulation v\'erifi\'ee pour $i+1$, et cherchons \`a
la montrer pour $i$.

On proc\`ede de la m\^eme mani\`ere que dans la preuve du lemme
\ref{annulation}. Par l'argument de filtration, on est ramen\'es \`a montrer l'annulation des trois groupes de cohomologie 
$H^{N-3-i}(Z,\Omega_Z^{2+i}(5+2i-N))$, $H^{N-3-i}(Z,\Omega_Z^{3+i}(7+2i-N))$ et $H^{N-2-i}(Z,\Omega^{3+i}_{\mathbb{P}^N_k}|_Z(7+2i-N))$. Le premier s'annule 
par hypoth\`ese de r\'ecurrence, le second s'annule par le lemme \ref{annulation}. Pour montrer l'annulation du troisi\`eme, on proc\`ede 
toujours comme dans la preuve du lemme
\ref{annulation} : en utilisant la suite spectrale de Koszul, on est ramen\'es \`a montrer l'annulation des trois groupes de cohomologie suivants :
$H^{N-2-i}(\mathbb{P}^N_k,\Omega_{\mathbb{P}^N_k}^{3+i}(7+2i-N))$, $H^{N-1-i}(\mathbb{P}^N_k,\Omega_{\mathbb{P}^N_k}^{3+i}(5+2i-N))$ et
$H^{N-i}(\mathbb{P}^N_k,\Omega_{\mathbb{P}^N_k}^{3+i}(3+2i-N))$. Pour cela, appliquons le lemme \ref{Bott}. Le cas (i) n'arrive que pour le troisi\`eme de ces groupes et 
$N=2i+3$, ce qui est impossible car $N$ est suppos\'e pair. Le cas (ii) n'intervient pas car $i<N-2$, et on v\'erifie facilement qu'on n'est jamais dans le cas (iii).
Cela conclut.
\end{proof}

On peut finalement montrer :

\begin{prop}\label{champdev}
Supposons qu'on n'a pas $c=1$ et $d_1=2$, ni $N=2$, $c=1$ et $d_1=3$, ni $N$ impair, $c=2$ et $d_1=d_2=2$. 
Soit $Z$ une intersection compl\`ete lisse sur $k$.

 Alors $H^0(Z,T_Z)=0$.
\end{prop}

\begin{proof}[$\mathbf{Preuve}$]
Par dualit\'e de Serre, $H^0(Z,T_Z)=H^n(Z,\Omega^1_Z(d_1+\ldots+d_c-N-1))^{\vee}$. Cherchons \`a annuler ce groupe de cohomologie \`a l'aide du lemme
\ref{annulation}. La premi\`ere condition $p+q=1+n>n$ est trivialement v\'erifi\'e. Pour la seconde, on remarque que $l+q-p=d_1+\ldots+d_c-c-2>0$ sauf si 
$c=1$ et $d_1=2$ ou $3$, ou si $c=2$ et $d_1=d_2=2$. 

Si $c=1$, $d_1=3$ et $N\geq 3$, on peut appliquer le lemme \ref{annulationbis}.  Si $c=2$, $d_1=d_2=2$ et $N$ est pair, on peut appliquer
le lemme \ref{annulationter}.
\end{proof}

\subsubsection{Intersections de deux quadriques}\label{partieicquadriques}

Dans ce paragraphe, on effectue des calculs sur les intersections de deux quadriques par manipulation explicite de leurs \'equations.

\begin{prop}\label{quadrdiago}
Soient $k$ un corps alg\'ebriquement clos de caract\'eristique $\neq 2$, $N\geq 3$ et
$Z\subset\mathbb{P}^N_k$ une intersection compl\`ete lisse de deux quadriques.
Alors l'espace tangent en $\Id$ de $\Aut_k(Z,\mathcal{O}(1))$ est trivial.
\end{prop}

\begin{proof}[$\mathbf{Preuve}$]
Comme $Z$ est lisse, on peut appliquer \cite{Wittenberg} Proposition
3.28 : on obtient un syst\`eme de coordonn\'ees dans lequel $Z=\{q=q'=0\}$ avec
$q=X_0^2+\dots+X_N^2$, $q'=a_0X_0^2+\dots+a_NX_N^2$, et o\`u les $a_i$ sont
distincts.

Notons $PG=\Aut_k(Z,\mathcal{O}(1))$ et $G$ son image r\'eciproque dans $GL_{N+1}$ : on a une suite exacte courte $0\to\mathbb{G}_m\to G\to PG\to 0$.
On note $A=k[\varepsilon]/\varepsilon^2$ les nombres duaux de sorte que l'espace tangent \`a $GL_{N+1}$ en $\Id$ s'identifie aux matrices de la forme
$(\Id+\varepsilon M)\in GL_{N+1}(A)$. Soit $v$ un vecteur tangent \`a $PG$ en $\Id$. Par lissit\'e de $G\to PG$, on le rel\`eve en
$g=\Id+\varepsilon M\in G(A)$ un vecteur tangent \`a $G$ en $\Id$. On note $(m_{i,j})_{0\leq i,j\leq N}\in k$
les coefficients de la matrice $M$.

Par la proposition \ref{propic} (iii), $H^0(\mathbb{P}^N_k,\mathcal{I}_Z(2))$ est de dimension $2$, engendr\'e par $q$ et $q'$. 
Ainsi, par changement de base par le morphisme 
plat $k\to A$, $H^0(\mathbb{P}^N_A,\mathcal{I}_{Z_A}(2))$ est un $A$-module libre de rang $2$, engendr\'e par $q$ et $q'$.
La matrice $g=\Id+\varepsilon M$ agit sur $H^0(\mathbb{P}^N_A,\mathcal{O}(2))$ en pr\'eservant ce sous-module, de sorte que $q\circ g$
et $q'\circ g$ sont combinaisons \`a coefficients dans $A$ de $q$ et $q'$. En calculant les termes constants (sans $\varepsilon$), on voit qu'il
existe $\alpha,\beta,\gamma,\delta\in k$ tels que ces
relations soient de la forme suivante :
\begin{equation}\label{relq}
q\circ g=(1+\varepsilon\alpha)q+\varepsilon\beta q'
\end{equation}
\begin{equation}\label{relq'}
q'\circ g=\varepsilon\gamma q+(1+\varepsilon\delta)q'
\end{equation}

Comme aucun mon\^ome $X_iX_j$ avec $i\neq j$ n'intervient dans le terme de droite de (\ref{relq}), on obtient $m_{i,j}+m_{j,i}=0$ pour $i\neq j$. 
Proc\'edant de m\^eme avec (\ref{relq'}), on obtient $a_im_{i,j}+a_jm_{j,i}=0$ pour $i\neq j$. 
Comme $a_i\neq a_j$ pour $i\neq j$, cela montre $m_{i,j}=0$ pour $i\neq j$.

La relation (\ref{relq}) s'\'ecrit alors $m_{i,i}=\alpha+\beta a_i$ pour $0\leq i\leq N$. De m\^eme, la relation
(\ref{relq'}) s'\'ecrit $a_im_{i,i}=\gamma+\delta a_i$ pour $0\leq i\leq N$. De ces deux relations, il vient que, pour
$0\leq i\leq N$, $\beta a_i^2+(\alpha-\delta)a_i-\gamma=0$. Les $a_i$ sont alors $N+1$ racines distinctes d'un polyn\^ome de 
degr\'e $2$. Ce polyn\^ome est donc nul : en particulier, $\beta=0$ et $m_{i,i}=\alpha$ pour tout $i$.

On a montr\'e que $M$ est une homoth\'etie, donc que $g$ est en fait tangent \`a $\mathbb{G}_m$. Par cons\'equent, $v=0$, et
l'espace tangent de $PG$ en $\Id$ est bien trivial.
\end{proof}

\begin{prop}\label{quadrdeux}
Soient $k$ un corps alg\'ebriquement clos de caract\'eristique $2$, $N\geq 3$ impair, et
$Z\subset\mathbb{P}^N_k$ une intersection compl\`ete lisse de deux quadriques.
Alors l'espace tangent en $\Id$ de $\Aut_k(Z,\mathcal{O}(1))$ est de dimension $\geq \frac{N-1}{2}$.
\end{prop}

\begin{proof}[$\mathbf{Preuve}$]

Soient $Z=\{q=q'=0\}$ des \'equations de $Z$. Notons $b$ et $b'$ les formes bilin\'eaires sym\'etriques associ\'ees aux formes quadratiques 
$q$ et $q'$. Comme la caract\'eristique de $k$ est $2$, le polyn\^ome $\det(\lambda b+\mu b')$, homog\`ene de degr\'e $N+1$ en $\lambda$ et $\mu$, est le
carr\'e du pfaffien $\pfaff(\lambda b+\mu b')$. Comme il suffit de montrer la proposition pour $Z$ g\'en\'erale, on peut supposer que les racines
de ce pfaffien sont distinctes.

Dans ce cas, on peut appliquer \cite{Bhosle} Coro. 2.10 : en notant $N+1=2r$, il existe un syst\`eme de coordonn\'ees dans lequel 
$Z=\{q=q'=0\}$ avec $q=\sum_{i=1}^r X_iY_i$ et $q'=\sum_{i=1}^r a_iX_iY_i+c_iX_i^2+d_iY_i^2$.

On raisonne alors comme dans la preuve de la proposition pr\'ec\'edente, dont on conserve les notations.
Un vecteur tangent \`a $GL_{N+1}$ en $\Id$ est un \'el\'ement $g=\Id+\varepsilon M\in GL_{N+1}(A)$. S'il pr\'eserve
le sous-$A$-module libre de rang $2$ de $H^0(\mathbb{P}^N_A,\mathcal{O}(2))$ engendr\'e par $q$ et $q'$, il pr\'eserve $Z_A=\{q=q'=0\}$ et est donc tangent
\`a $G$ en $\Id$. Or, si $D\in M_r(k)$ est diagonale, et si on note
\[
 M=
 \begin{pmatrix}
  D & 0  \\
  0 & D 
 \end{pmatrix},
\]
on v\'erifie par calcul que $g=\Id+\varepsilon M$ pr\'eserve ce sous-module, de sorte que $g$ est tangent \`a $G$ en $\Id$. Ceci montre que 
l'espace de tangent de $G$ en $\Id$ est de dimension $\geq r=\frac{N+1}{2}$. De la suite exacte $0\to\mathbb{G}_m\to G\to PG\to 0$, on
d\'eduit que l'espace tangent de $PG$ en $\Id$ est de dimension $\geq \frac{N-1}{2}$, comme voulu.
\end{proof}

\subsubsection{Fin de la preuve}\label{preuveautoic}

Montrons finalement le th\'eor\`eme \ref{autoic}.

\begin{proof}[$\mathbf{Preuve \text{ }du \text{ }th\acute{e}or\grave{e}me\text{ }\ref{autoic}}$]~

On se ram\`ene au cas o\`u $k$ est alg\'ebrique\-ment clos ; pour cela, on note $\bar{k}$ une cl\^oture alg\'ebrique de $k$.
Par description de leurs foncteurs des points,
$\Aut_{\bar{k}}(Z_{\bar{k}})=\Aut_k(Z)_{\bar{k}}$. Ainsi, le $k$-sch\'ema en groupes $\Aut_k(Z)$ est de dimension nulle
(resp. fini, resp. lisse, resp. de dimension
nulle et r\'eduit) 
si et seulement si le $\bar{k}$-sch\'ema en groupes $\Aut_{\bar{k}}(Z_{\bar{k}})$ l'est. 
Le seul point non trivial dans cette assertion est le fait que si $\Aut_k(Z)$ est r\'eduit de dimension nulle, $\Aut_{\bar{k}}(Z_{\bar{k}})$ est
r\'eduit. Mais si c'est le cas, la composante connexe de l'identit\'e de $\Aut_k(Z)$ est un $k$-sch\'ema connexe de dimension nulle donc ponctuel, 
r\'eduit donc isomorphe au spectre d'un corps, avec un $k$-point donc $k$-isomorphe \`a $\Spec(k)$.
Par cons\'equent, la composante connexe de l'identit\'e de $\Aut_{\bar{k}}(Z_{\bar{k}})$ est $\Spec(\bar{k})$. Par homog\'en\'eit\'e sous l'action des $\bar{k}$-points
de $\Aut_{\bar{k}}(Z_{\bar{k}})$, toutes ses composantes connexes sont isomorphes \`a $\Spec(\bar{k})$, donc r\'eduites.
Le m\^eme raisonnement permet de comparer $\Aut_k(Z,\mathcal{O}(1))$ et $\Aut_{\bar{k}}(Z_{\bar{k}},\mathcal{O}(1))$. Dans la suite, on suppose donc
$k$ al\-g\'e\-bri\-que\-ment clos.

Commen\c{c}ons par montrer que si $n\geq 2$ et si l'on n'est pas dans le cas (iv), $\Aut_k(Z)=\Aut_k(Z,\mathcal{O}(1))$. Pour cela, soient $T$ un $k$-sch\'ema et 
$f\in \Aut_k(Z)(T)$ un $T$-automorphisme de $Z_T$. Il faut montrer que le diagramme (\ref{diagpic}) commute. Raisonnant composante connexe par composante connexe, 
on peut supposer $T$ connexe. Remarquons tout d'abord que, comme 
$n\geq 2$, la proposition \ref{propic} (ii) montre que 
l'espace tangent en l'identit\'e $H^1(Z,\mathcal{O}_Z)$ \`a $\mathbf{Pic}_{Z/k}$ est nul, de sorte que $\mathbf{Pic}_{Z_T/T}$ est r\'eunion
de composantes connexes $T$-isomorphes \`a $T$. Les sections $\lambda_T$ et $f^*\circ\lambda_T$ sont n\'ecessairement des isomorphismes sur l'une de ces composantes 
connexes : ainsi, pour qu'elles co\"incident, il suffit qu'elles co\"incident en un point g\'eom\'etrique. On est donc ramen\'es \`a v\'erifier que, sous
nos hypoth\`eses, un automorphisme $f$
d'une intersection compl\`ete lisse $Z$ sur un corps alg\'ebriquement clos pr\'eserve $\mathcal{O}(1)$. Si $n\geq 3$, 
par th\'eor\`eme de Lefschetz, $\mathcal{O}(1)$ est 
l'unique g\'en\'erateur ample du groupe de Picard de $Z$, et est donc pr\'eserv\'e par $f$. Si $n=2$, mais qu'on n'est pas dans le cas (iv), 
\cite{SGA7} Expos\'e XI, Th. 1.8
montre que $\mathcal{O}(1)$ est l'unique g\'en\'erateur ample du sous-groupe du groupe de Picard de $Z$ constitu\'e des fibr\'es en droites dont un multiple
est proportionnel au diviseur canonique ; cette caract\'erisation montre qu'il est pr\'eserv\'e par $f$.

D'autre part, par la proposition \ref{champdev}, si l'on n'est pas dans un des cas (i), (ii) ou (v), $H^0(Z,T_Z)=0$. Comme cet espace vectoriel s'identifie
\`a l'espace
tangent en l'identit\'e de $\Aut_k(Z)$, $\Aut_k(Z)$ est alors un sch\'ema en groupes r\'eduit de dimension $0$. 

Ces deux faits, combin\'es aux r\'esultats g\'en\'eraux du paragraphe \ref{parauto}, 
montrent le th\'eor\`eme 
sauf dans les cas (i), (ii), (iii) et (v) qu'on discute \`a pr\'esent.

Le cas (i) des quadriques est classique : la composante connexe de $\Aut_k(Z)$ est un groupe semi-simple de type orthogonal.

Le cas (iii) des courbes de genre $\geq 2$ est \'egalement classique, mais donnons un argument. 
Comme $T_Z$ est un fibr\'e en droites anti-ample, $H^0(Z,T_Z)=0$, de sorte que $\Aut_k(Z)$
est r\'eduit de dimension $0$. Pour montrer que ce sch\'ema en groupes est fini, on peut remarquer que, comme un automorphisme pr\'eserve $3K_Z$,
$\Aut_k(Z)=\Aut_k(Z,3K_Z)$. Alors, en consid\'erant le plongement 
tricanonique de $Z$ et en
proc\'edant comme au paragraphe \ref{autoproj}, 
on r\'ealise $\Aut_k(Z)$ comme un sous-sch\'ema en groupes d'un groupe lin\'eaire, de sorte qu'il est de type fini, donc fini.
Enfin, $\Aut_k(Z,\mathcal{O}(1))$ est fini r\'eduit comme sous-sch\'ema en groupes de $\Aut_k(Z)$.

Dans le cas (ii), $Z$ est une courbe de genre $1$. La courbe elliptique $E$ sous-jacente agit par translations sur $Z$. 
Comme $h^0(Z,T_Z)=h^0(Z,\mathcal{O}_Z)=1$, on voit que $E$ s'identifie \`a la composante connexe de l'identit\'e de $\Aut_k(Z)$. Finalement, le groupe des
composantes connexes de $\Aut_k(Z)$ est fini par \cite{Silverman}, Chapter III, Theorem 10.1.
Il reste \`a d\'eterminer le sch\'ema en groupes $E\cap \Aut_k(Z,\mathcal{O}(1))$. Il est expliqu\'e dans \cite{KatzSarnak} 11.7 p. 342 que c'est $E[3]$
(resp. $E[4]$) si $N=2$, $c=1$ et $d_1=3$ (resp. si
$N=3$, $c=2$ et $d_1=d_2=2$). Le Corollary 6.4 de \cite{Silverman} permet de comparer le degr\'e et le cardinal de ce sous-groupe, et donc de v\'erifier qu'il
est non r\'eduit si et seulement si $p=3$ (resp. $p=2$).

Traitons enfin le cas (v). On note $H$ le sch\'ema de Hilbert des intersections compl\`etes lisses.
Le th\'eor\`eme \ref{principal} sous sa forme (iii) montre la propret\'e de l'action de $PGL_{N+1}$ sur $H$.
La description de $\Aut_k(Z,\mathcal{O}(1))$ comme sous-groupe 
de $PGL_{N+1}$ montre qu'il s'identifie au stabilisateur de $[Z]\in H(k)$ pour cette action : il est donc propre. 
Comme il est affine comme sous-groupe ferm\'e d'un groupe affine, 
il est fini. La proposition \ref{quadrdiago} montre que si $p\neq 2$, l'espace tangent en l'identit\'e de 
$\Aut_k(Z,\mathcal{O}(1))$ est trivial, de sorte que ce sch\'ema en groupes est fini et r\'eduit. 
La proposition \ref{quadrdeux} montre, elle, que si $p=2$, l'espace tangent en l'identit\'e de ce groupe est non trivial,
de sorte que ce sch\'ema en groupes est fini mais
non r\'eduit.
\end{proof}

\subsection{Propri\'et\'e de Deligne-Mumford}\label{partieDM}

On peut enfin montrer :

\begin{thm}\label{icDM}

Le champ $\mathcal{M}$ est de Deligne-Mumford sauf dans les cas suivants :

\begin{enumerate}[(i)]
 \item Si $c=1$ et $d_1=2$.

\item Si $N=2$, $c=1$, $d_1=3$, auquel cas il est de Deligne-Mumford au-dessus de $\Spec(\mathbb{Z}[\frac{1}{3}])$.

\item Si $N\geq 3$ est impair, $c=2$, $d_1=d_2=2$, auquel cas il est de Deligne-Mumford au-dessus de $\Spec(\mathbb{Z}[\frac{1}{2}])$.
\end{enumerate}
 
\end{thm}

\begin{proof}[$\mathbf{Preuve}$]
Soient $k$ un corps alg\'ebriquement clos et $Z$ une intersection compl\`ete lisse sur $k$.
La description de $\Aut_k(Z,\mathcal{O}(1))$ comme sous-groupe 
de $PGL_{N+1}$ montre qu'il s'identifie au stabilisateur de $[Z]\in H(k)$ pour l'action de $PGL_{N+1}$ par changement de coordonn\'ees.
Comme le champ $\mathcal{M}$ est de Deligne-Mumford si et seulement si les stabilisateurs g\'eom\'etriques de l'action de $PGL_{N+1}$ sur $H$
sont finis et r\'eduits, le th\'eor\`eme \ref{autoic} permet de montrer la proposition.
\end{proof}

\addcontentsline{toc}{section}{R\'{e}f\'{e}rences}

\bibliographystyle{plain-fr}
\bibliography{biblio}

\end{document}